\documentclass[12pt]{amsart}
\usepackage[latin1]{inputenc}
\usepackage{amsbsy}
\usepackage{amscd} 
\usepackage{amsfonts}
\usepackage{amsgen} 
\usepackage{amsmath}
\usepackage{amsopn} 
\usepackage{amssymb}
\usepackage{amstext}
\usepackage{amsthm} 
\usepackage{amsxtra}
\usepackage[all]{xy}
\usepackage{epsfig}
\usepackage{mathtools}
\usepackage{rotating}

\usepackage{calc}
\newcounter{PartitionDepth}
\newcounter{PartitionLength}
\newcommand{\parti}[2]{
 \begin{picture}(#2,#1)
 \setcounter{PartitionDepth}{-1-#1}
 \put(#2,\thePartitionDepth){\line(0,1){#1}}
 \end{picture}}
\newcommand{\partii}[3]{
 \begin{picture}(#3,#1)
 \setcounter{PartitionLength}{#3-#2}
 \setcounter{PartitionDepth}{-1-#1}
 \put(#2,\thePartitionDepth){\line(0,1){#1}}     
 \put(#3,\thePartitionDepth){\line(0,1){#1}}
 \put(#2,\thePartitionDepth){\line(1,0){\thePartitionLength}}
 \end{picture}}
\newcommand{\partiii}[4]{
 \begin{picture}(#4,#1)
 \setcounter{PartitionLength}{#4-#2}
 \setcounter{PartitionDepth}{-1-#1}
 \put(#2,\thePartitionDepth){\line(0,1){#1}}
 \put(#3,\thePartitionDepth){\line(0,1){#1}}
 \put(#4,\thePartitionDepth){\line(0,1){#1}}
 \put(#2,\thePartitionDepth){\line(1,0){\thePartitionLength}} 
 \end{picture}}

\newcommand{\upparti}[2]{
 \begin{picture}(#2,#1)
 \setcounter{PartitionDepth}{#1}
 \put(#2,0){\line(0,1){#1}}
 \end{picture}}
\newcommand{\uppartii}[3]{
 \begin{picture}(#3,#1)
 \setcounter{PartitionLength}{#3-#2}
 \setcounter{PartitionDepth}{#1}
 \put(#2,0){\line(0,1){#1}}     
 \put(#3,0){\line(0,1){#1}}
 \put(#2,\thePartitionDepth){\line(1,0){\thePartitionLength}}
 \end{picture}}
\newcommand{\uppartiii}[4]{
 \begin{picture}(#4,#1)
 \setcounter{PartitionLength}{#4-#2}
 \setcounter{PartitionDepth}{#1}
 \put(#2,0){\line(0,1){#1}}
 \put(#3,0){\line(0,1){#1}}
 \put(#4,0){\line(0,1){#1}}
 \put(#2,\thePartitionDepth){\line(1,0){\thePartitionLength}} 
 \end{picture}}
\newcommand{\uppartiv}[5]{
 \begin{picture}(#5,#1)
 \setcounter{PartitionLength}{#5-#2}
 \setcounter{PartitionDepth}{#1}
 \put(#2,0){\line(0,1){#1}}
 \put(#3,0){\line(0,1){#1}}
 \put(#4,0){\line(0,1){#1}}
 \put(#5,0){\line(0,1){#1}}
 \put(#2,\thePartitionDepth){\line(1,0){\thePartitionLength}} 
 \end{picture}}

\newsavebox{\boxpaarpart}
   \savebox{\boxpaarpart}
   { \begin{picture}(0.7,0.5)
     \put(0,0){\line(0,1){0.4}}
     \put(0.5,0){\line(0,1){0.4}}
     \put(0,0.4){\line(1,0){0.5}}
     \end{picture}}
\newsavebox{\boxbaarpart}
   \savebox{\boxbaarpart}
   { \begin{picture}(0.7,0.5)
     \put(0,0){\line(0,1){0.4}}
     \put(0.5,0){\line(0,1){0.4}}
     \put(0,0){\line(1,0){0.5}}
     \end{picture}}
\newsavebox{\boxdreipart}
   \savebox{\boxdreipart}
   { \begin{picture}(1,0.5)
     \put(0,0){\line(0,1){0.4}}
     \put(0.4,0){\line(0,1){0.4}}
     \put(0.8,0){\line(0,1){0.4}}
     \put(0,0.4){\line(1,0){0.8}}
     \end{picture}}
\newsavebox{\boxvierpart}
   \savebox{\boxvierpart}
   { \begin{picture}(1.4,0.5)
     \put(0,0){\line(0,1){0.4}}
     \put(0.4,0){\line(0,1){0.4}}
     \put(0.8,0){\line(0,1){0.4}}
     \put(1.2,0){\line(0,1){0.4}}
     \put(0,0.4){\line(1,0){1.2}}
     \end{picture}}
\newsavebox{\boxvierpartrot}
   \savebox{\boxvierpartrot}
   { \begin{picture}(0.5,0.8)
     \put(0,-0.2){\line(0,1){0.3}}
     \put(0.4,-0.2){\line(0,1){0.3}}
     \put(0,0.1){\line(1,0){0.4}}
     \put(0,0.4){\line(0,1){0.3}}
     \put(0.4,0.4){\line(0,1){0.3}}
     \put(0,0.4){\line(1,0){0.4}}
     \put(0.2,0.1){\line(0,1){0.3}}
     \end{picture}}
\newsavebox{\boxvierpartrotdrei}
   \savebox{\boxvierpartrotdrei}
   { \begin{picture}(1,0.8)
     \put(0,-0.2){\line(0,1){0.3}}
     \put(0.4,-0.2){\line(0,1){0.8}}
     \put(0.8,-0.2){\line(0,1){0.3}}
     \put(0,0.1){\line(1,0){0.8}}
     \end{picture}}
\newsavebox{\boxcrosspart}
   \savebox{\boxcrosspart}
   { \begin{picture}(0.5,0.8)
     \put(0,-0.2){\line(1,2){0.4}}
     \put(0.4,-0.2){\line(-1,2){0.4}}
     \end{picture}}
\newsavebox{\boxhalflibpart}
   \savebox{\boxhalflibpart}
   { \begin{picture}(1,0.8)
     \put(0,-0.2){\line(1,1){0.8}}
     \put(0.8,-0.2){\line(-1,1){0.8}}
     \put(0.4,-0.2){\line(0,1){0.8}}
     \end{picture}}
\newsavebox{\boxpositioner} 
   \savebox{\boxpositioner}
   { \begin{picture}(1.4,0.5)
     \put(0,0){\line(0,1){0.3}}
     \put(0.4,0){\line(0,1){0.6}}
     \put(0.8,0){\line(0,1){0.3}}
     \put(1.2,0){\line(0,1){0.6}}
     \put(0.4,0.6){\line(1,0){0.8}}
     \end{picture}}
\newsavebox{\boxfatcross} 
   \savebox{\boxfatcross}
   { \begin{picture}(1.4,1.3)
     \put(0,-0.2){\line(0,1){0.3}}
     \put(0.4,-0.2){\line(0,1){0.3}}
     \put(0.8,-0.2){\line(0,1){0.3}}
     \put(1.2,-0.2){\line(0,1){0.3}}
     \put(0,0.1){\line(1,0){0.4}}
     \put(0.8,0.1){\line(1,0){0.4}}
     \put(0,0.9){\line(0,1){0.3}}
     \put(0.4,0.9){\line(0,1){0.3}}
     \put(0.8,0.9){\line(0,1){0.3}}
     \put(1.2,0.9){\line(0,1){0.3}}
     \put(0,0.9){\line(1,0){0.4}}
     \put(0.8,0.9){\line(1,0){0.4}}
     \put(0.2,0.1){\line(1,1){0.8}}
     \put(0.2,0.9){\line(1,-1){0.8}}
     \end{picture}}
\newsavebox{\boxprimarypart} 
   \savebox{\boxprimarypart}
   { \begin{picture}(1,1.3)
     \put(0,-0.2){\line(0,1){0.3}}
     \put(0.4,-0.2){\line(0,1){0.3}}
     \put(0.8,-0.2){\line(0,1){0.3}}
     \put(0.4,0.1){\line(1,0){0.4}}
     \put(0,0.9){\line(0,1){0.3}}
     \put(0.4,0.9){\line(0,1){0.3}}
     \put(0.8,0.9){\line(0,1){0.3}}
     \put(0,0.9){\line(1,0){0.4}}
     \put(0,0.1){\line(1,1){0.8}}
     \put(0.2,0.9){\line(1,-2){0.4}}
     \end{picture}}

\newsavebox{\boxidpartww}
   \savebox{\boxidpartww}
   { \begin{picture}(0.5,1.2)
     \put(0.2,0.15){\line(0,1){0.7}}
     \put(0,-0.2){$\circ$}
     \put(0,0.8){$\circ$}
     \end{picture}}
\newsavebox{\boxidpartbw}
   \savebox{\boxidpartbw}
   { \begin{picture}(0.5,1.2)
     \put(0.2,0.15){\line(0,1){0.7}}
     \put(0,-0.2){$\circ$}
     \put(0,0.8){$\bullet$}
     \end{picture}}
\newsavebox{\boxidpartwb}
   \savebox{\boxidpartwb}
   { \begin{picture}(0.5,1.2)
     \put(0.2,0.15){\line(0,1){0.7}}
     \put(0,-0.2){$\bullet$}
     \put(0,0.8){$\circ$}
     \end{picture}}
\newsavebox{\boxidpartbb}
   \savebox{\boxidpartbb}
   { \begin{picture}(0.5,1.2)
     \put(0.2,0.15){\line(0,1){0.7}}
     \put(0,-0.2){$\bullet$}
     \put(0,0.8){$\bullet$}
     \end{picture}}

\newsavebox{\boxidpartsingletonbb}
   \savebox{\boxidpartsingletonbb}
   { \begin{picture}(0.5,1.2)
     \put(0.2,0.15){\line(0,1){0.2}}
     \put(0.2,0.65){\line(0,1){0.2}}
     \put(0,-0.2){$\bullet$}
     \put(0,0.8){$\bullet$}
     \end{picture}}
\newsavebox{\boxidpartsingletonww}
   \savebox{\boxidpartsingletonww}
   { \begin{picture}(0.5,1.2)
     \put(0.2,0.15){\line(0,1){0.2}}
     \put(0.2,0.65){\line(0,1){0.2}}
     \put(0,-0.2){$\circ$}
     \put(0,0.8){$\circ$}
     \end{picture}}
     
\newsavebox{\boxpaarpartbb}
   \savebox{\boxpaarpartbb}
   { \begin{picture}(1,1)
     \put(0.2,0.2){\line(0,1){0.4}}
     \put(0.7,0.2){\line(0,1){0.4}}
     \put(0.2,0.6){\line(1,0){0.5}}
     \put(0.5,-0.2){$\bullet$}
     \put(0,-0.2){$\bullet$}
     \end{picture}}
\newsavebox{\boxpaarpartww}
   \savebox{\boxpaarpartww}
   { \begin{picture}(1,1)
     \put(0.2,0.2){\line(0,1){0.4}}
     \put(0.7,0.2){\line(0,1){0.4}}
     \put(0.2,0.6){\line(1,0){0.5}}
     \put(0.5,-0.2){$\circ$}
     \put(0,-0.2){$\circ$}
     \end{picture}}
\newsavebox{\boxpaarpartbw}
   \savebox{\boxpaarpartbw}
   { \begin{picture}(1,1)
     \put(0.2,0.2){\line(0,1){0.4}}
     \put(0.7,0.2){\line(0,1){0.4}}
     \put(0.2,0.6){\line(1,0){0.5}}
     \put(0.5,-0.2){$\circ$}
     \put(0,-0.2){$\bullet$}
     \end{picture}}
\newsavebox{\boxpaarpartwb}
   \savebox{\boxpaarpartwb}
   { \begin{picture}(1,1)
     \put(0.2,0.2){\line(0,1){0.4}}
     \put(0.7,0.2){\line(0,1){0.4}}
     \put(0.2,0.6){\line(1,0){0.5}}
     \put(0.5,-0.2){$\bullet$}
     \put(0,-0.2){$\circ$}
     \end{picture}}
\newsavebox{\boxbaarpartbb}
   \savebox{\boxbaarpartbb}
   { \begin{picture}(1,1)
     \put(0.2,-0.1){\line(0,1){0.4}}
     \put(0.7,-0.1){\line(0,1){0.4}}
     \put(0.2,-0.1){\line(1,0){0.5}}
     \put(0.5,0.3){$\bullet$}
     \put(0,0.3){$\bullet$}
     \end{picture}}
\newsavebox{\boxbaarpartww}
   \savebox{\boxbaarpartww}
   { \begin{picture}(1,1)
     \put(0.2,-0.1){\line(0,1){0.4}}
     \put(0.7,-0.1){\line(0,1){0.4}}
     \put(0.2,-0.1){\line(1,0){0.5}}
     \put(0.5,0.3){$\circ$}
     \put(0,0.3){$\circ$}
     \end{picture}}
\newsavebox{\boxbaarpartbw}
   \savebox{\boxbaarpartbw}
   { \begin{picture}(1,1)
     \put(0.2,-0.1){\line(0,1){0.4}}
     \put(0.7,-0.1){\line(0,1){0.4}}
     \put(0.2,-0.1){\line(1,0){0.5}}
     \put(0.5,0.3){$\circ$}
     \put(0,0.3){$\bullet$}
     \end{picture}}
\newsavebox{\boxbaarpartwb}
   \savebox{\boxbaarpartwb}
   { \begin{picture}(1,1)
     \put(0.2,-0.1){\line(0,1){0.4}}
     \put(0.7,-0.1){\line(0,1){0.4}}
     \put(0.2,-0.1){\line(1,0){0.5}}
     \put(0.5,0.3){$\bullet$}
     \put(0,0.3){$\circ$}
     \end{picture}}
\newsavebox{\boxpaarbaarpartwwww}
   \savebox{\boxpaarbaarpartwwww}
   { \begin{picture}(1,1.2)
     \put(0.2,0.15){\line(0,1){0.2}}
     \put(0.2,0.65){\line(0,1){0.2}}
     \put(0.2,0.65){\line(1,0){0.5}}
     \put(0.2,0.35){\line(1,0){0.5}}
     \put(0.7,0.15){\line(0,1){0.2}}
     \put(0.7,0.65){\line(0,1){0.2}}
     \put(0,0.8){$\circ$}
     \put(0.5,0.8){$\circ$}
     \put(0,-0.2){$\circ$}
     \put(0.5,-0.2){$\circ$}
     \end{picture}}     

\newsavebox{\boxsingletonw}
   \savebox{\boxsingletonw}
   { \begin{picture}(0.5, 1)
     \put(0,0.3){$\uparrow$}
     \put(0,-0.2){$\circ$}
     \end{picture}}
\newsavebox{\boxsingletonb}
   \savebox{\boxsingletonb}
   { \begin{picture}(0.5, 1)
     \put(0,0.3){$\uparrow$}
     \put(0,-0.2){$\bullet$}
     \end{picture}}
\newsavebox{\boxdownsingletonw}
   \savebox{\boxdownsingletonw}
   { \begin{picture}(0.5, 1)
     \put(0,0){$\downarrow$}
     \put(0,0.6){$\circ$}
     \end{picture}}
\newsavebox{\boxdownsingletonb}
   \savebox{\boxdownsingletonb}
   { \begin{picture}(0.5, 1)
     \put(0,0){$\downarrow$}
     \put(0,0.6){$\bullet$}
     \end{picture}}

\newsavebox{\boxvierpartwwww}
   \savebox{\boxvierpartwwww}
   { \begin{picture}(1.7,1)
     \put(0.2,0.2){\line(0,1){0.4}}
     \put(0.6,0.2){\line(0,1){0.4}}
     \put(1,0.2){\line(0,1){0.4}}
     \put(1.4,0.2){\line(0,1){0.4}}
     \put(0.2,0.6){\line(1,0){1.2}}
     \put(0,-0.2){$\circ$}
     \put(0.4,-0.2){$\circ$}
     \put(0.8,-0.2){$\circ$}
     \put(1.2,-0.2){$\circ$}
     \end{picture}}
\newsavebox{\boxdownvierpartwwww}
   \savebox{\boxdownvierpartwwww}
   { \begin{picture}(1.7,1)
     \put(0.2,0.2){\line(0,1){0.4}}
     \put(0.6,0.2){\line(0,1){0.4}}
     \put(1,0.2){\line(0,1){0.4}}
     \put(1.4,0.2){\line(0,1){0.4}}
     \put(0.2,0.2){\line(1,0){1.2}}
     \put(0,0.5){$\circ$}
     \put(0.4,0.5){$\circ$}
     \put(0.8,0.5){$\circ$}
     \put(1.2,0.5){$\circ$}
     \end{picture}}     
\newsavebox{\boxvierpartwbwb}
   \savebox{\boxvierpartwbwb}
   { \begin{picture}(1.7,1)
     \put(0.2,0.2){\line(0,1){0.4}}
     \put(0.6,0.2){\line(0,1){0.4}}
     \put(1,0.2){\line(0,1){0.4}}
     \put(1.4,0.2){\line(0,1){0.4}}
     \put(0.2,0.6){\line(1,0){1.2}}
     \put(0,-0.2){$\circ$}
     \put(0.4,-0.2){$\bullet$}
     \put(0.8,-0.2){$\circ$}
     \put(1.2,-0.2){$\bullet$}
     \end{picture}}
\newsavebox{\boxvierpartbwbw}
   \savebox{\boxvierpartbwbw}
   { \begin{picture}(1.7,1)
     \put(0.2,0.2){\line(0,1){0.4}}
     \put(0.6,0.2){\line(0,1){0.4}}
     \put(1,0.2){\line(0,1){0.4}}
     \put(1.4,0.2){\line(0,1){0.4}}
     \put(0.2,0.6){\line(1,0){1.2}}
     \put(0,-0.2){$\bullet$}
     \put(0.4,-0.2){$\circ$}
     \put(0.8,-0.2){$\bullet$}
     \put(1.2,-0.2){$\circ$}
     \end{picture}}
\newsavebox{\boxvierpartwwbb}
   \savebox{\boxvierpartwwbb}
   { \begin{picture}(1.7,1)
     \put(0.2,0.2){\line(0,1){0.4}}
     \put(0.6,0.2){\line(0,1){0.4}}
     \put(1,0.2){\line(0,1){0.4}}
     \put(1.4,0.2){\line(0,1){0.4}}
     \put(0.2,0.6){\line(1,0){1.2}}
     \put(0,-0.2){$\circ$}
     \put(0.4,-0.2){$\circ$}
     \put(0.8,-0.2){$\bullet$}
     \put(1.2,-0.2){$\bullet$}
     \end{picture}}
\newsavebox{\boxvierpartwbbw}
   \savebox{\boxvierpartwbbw}
   { \begin{picture}(1.7,1)
     \put(0.2,0.2){\line(0,1){0.4}}
     \put(0.6,0.2){\line(0,1){0.4}}
     \put(1,0.2){\line(0,1){0.4}}
     \put(1.4,0.2){\line(0,1){0.4}}
     \put(0.2,0.6){\line(1,0){1.2}}
     \put(0,-0.2){$\circ$}
     \put(0.4,-0.2){$\bullet$}
     \put(0.8,-0.2){$\bullet$}
     \put(1.2,-0.2){$\circ$}
     \end{picture}}
\newsavebox{\boxvierpartbwwb}
   \savebox{\boxvierpartbwwb}
   { \begin{picture}(1.7,1)
     \put(0.2,0.2){\line(0,1){0.4}}
     \put(0.6,0.2){\line(0,1){0.4}}
     \put(1,0.2){\line(0,1){0.4}}
     \put(1.4,0.2){\line(0,1){0.4}}
     \put(0.2,0.6){\line(1,0){1.2}}
     \put(0,-0.2){$\bullet$}
     \put(0.4,-0.2){$\circ$}
     \put(0.8,-0.2){$\circ$}
     \put(1.2,-0.2){$\bullet$}
     \end{picture}}
\newsavebox{\boxvierpartrotwbwb}
   \savebox{\boxvierpartrotwbwb}
   { \begin{picture}(1,1.2)
     \put(0.2,0.15){\line(0,1){0.2}}
     \put(0.2,0.65){\line(0,1){0.2}}
     \put(0.2,0.65){\line(1,0){0.5}}
     \put(0.2,0.35){\line(1,0){0.5}}
     \put(0.45,0.35){\line(0,1){0.3}}
     \put(0.7,0.15){\line(0,1){0.2}}
     \put(0.7,0.65){\line(0,1){0.2}}
     \put(0,0.8){$\circ$}
     \put(0.5,0.8){$\bullet$}
     \put(0,-0.2){$\circ$}
     \put(0.5,-0.2){$\bullet$}
     \end{picture}}
\newsavebox{\boxvierpartrotbwwb}
   \savebox{\boxvierpartrotbwwb}
   { \begin{picture}(1,1.2)
     \put(0.2,0.15){\line(0,1){0.2}}
     \put(0.2,0.65){\line(0,1){0.2}}
     \put(0.2,0.65){\line(1,0){0.5}}
     \put(0.2,0.35){\line(1,0){0.5}}
     \put(0.45,0.35){\line(0,1){0.3}}
     \put(0.7,0.15){\line(0,1){0.2}}
     \put(0.7,0.65){\line(0,1){0.2}}
     \put(0,0.8){$\bullet$}
     \put(0.5,0.8){$\circ$}
     \put(0,-0.2){$\circ$}
     \put(0.5,-0.2){$\bullet$}
     \end{picture}}
\newsavebox{\boxvierpartrotbwbw}
   \savebox{\boxvierpartrotbwbw}
   { \begin{picture}(1,1.2)
     \put(0.2,0.15){\line(0,1){0.2}}
     \put(0.2,0.65){\line(0,1){0.2}}
     \put(0.2,0.65){\line(1,0){0.5}}
     \put(0.2,0.35){\line(1,0){0.5}}
     \put(0.45,0.35){\line(0,1){0.3}}
     \put(0.7,0.15){\line(0,1){0.2}}
     \put(0.7,0.65){\line(0,1){0.2}}
     \put(0,0.8){$\bullet$}
     \put(0.5,0.8){$\circ$}
     \put(0,-0.2){$\bullet$}
     \put(0.5,-0.2){$\circ$}
     \end{picture}}
\newsavebox{\boxvierpartrotwwww}
   \savebox{\boxvierpartrotwwww}
   { \begin{picture}(1,1.2)
     \put(0.2,0.15){\line(0,1){0.2}}
     \put(0.2,0.65){\line(0,1){0.2}}
     \put(0.2,0.65){\line(1,0){0.5}}
     \put(0.2,0.35){\line(1,0){0.5}}
     \put(0.45,0.35){\line(0,1){0.3}}
     \put(0.7,0.15){\line(0,1){0.2}}
     \put(0.7,0.65){\line(0,1){0.2}}
     \put(0,0.8){$\circ$}
     \put(0.5,0.8){$\circ$}
     \put(0,-0.2){$\circ$}
     \put(0.5,-0.2){$\circ$}
     \end{picture}}
\newsavebox{\boxvierpartrotbbbb}
   \savebox{\boxvierpartrotbbbb}
   { \begin{picture}(1,1.2)
     \put(0.2,0.15){\line(0,1){0.2}}
     \put(0.2,0.65){\line(0,1){0.2}}
     \put(0.2,0.65){\line(1,0){0.5}}
     \put(0.2,0.35){\line(1,0){0.5}}
     \put(0.45,0.35){\line(0,1){0.3}}
     \put(0.7,0.15){\line(0,1){0.2}}
     \put(0.7,0.65){\line(0,1){0.2}}
     \put(0,0.8){$\bullet$}
     \put(0.5,0.8){$\bullet$}
     \put(0,-0.2){$\bullet$}
     \put(0.5,-0.2){$\bullet$}
     \end{picture}}

\newsavebox{\boxdreipartwww}
   \savebox{\boxdreipartwww}
   { \begin{picture}(1.2,1)
     \put(0.2,0.2){\line(0,1){0.4}}
     \put(0.6,0.2){\line(0,1){0.4}}
     \put(1,0.2){\line(0,1){0.4}}
     \put(0.2,0.6){\line(1,0){0.8}}
     \put(0,-0.2){$\circ$}
     \put(0.4,-0.2){$\circ$}
     \put(0.8,-0.2){$\circ$}
     \end{picture}}
\newsavebox{\boxdowndreipartwww}
   \savebox{\boxdowndreipartwww}
   { \begin{picture}(1.2,1)
     \put(0.2,-0.1){\line(0,1){0.4}}
     \put(0.6,-0.1){\line(0,1){0.4}}
     \put(1,-0.1){\line(0,1){0.4}}
     \put(0.2,-0.1){\line(1,0){0.8}}
     \put(0,0.3){$\circ$}
     \put(0.4,0.3){$\circ$}
     \put(0.8,0.3){$\circ$}
     \end{picture}}     
\newsavebox{\boxdreipartrotwww}
   \savebox{\boxdreipartrotwww}
   { \begin{picture}(1,1.2)
     \put(0.2,0.65){\line(0,1){0.2}}
     \put(0.2,0.65){\line(1,0){0.5}}
     \put(0.45,0.15){\line(0,1){0.5}}
     \put(0.7,0.65){\line(0,1){0.2}}
     \put(0,0.8){$\circ$}
     \put(0.5,0.8){$\circ$}
     \put(0.25,-0.2){$\circ$}
     \end{picture}}     
\newsavebox{\boxdowndreipartrotwww}
   \savebox{\boxdowndreipartrotwww}
   { \begin{picture}(1,1.2)
     \put(0.2,0.15){\line(0,1){0.2}}
     \put(0.2,0.35){\line(1,0){0.5}}
     \put(0.45,0.35){\line(0,1){0.5}}
     \put(0.7,0.15){\line(0,1){0.2}}
     \put(0,-0.2){$\circ$}
     \put(0.5,-0.2){$\circ$}
     \put(0.25,0.8){$\circ$}
     \end{picture}}          

\newsavebox{\boxsechspartwbwbwb}
   \savebox{\boxsechspartwbwbwb}
   { \begin{picture}(2.5,1)
     \put(0.2,0.2){\line(0,1){0.4}}
     \put(0.6,0.2){\line(0,1){0.4}}
     \put(1,0.2){\line(0,1){0.4}}
     \put(1.4,0.2){\line(0,1){0.4}}
     \put(1.8,0.2){\line(0,1){0.4}}
     \put(2.2,0.2){\line(0,1){0.4}}
     \put(0.2,0.6){\line(1,0){2}}
     \put(0,-0.2){$\circ$}
     \put(0.4,-0.2){$\bullet$}
     \put(0.8,-0.2){$\circ$}
     \put(1.2,-0.2){$\bullet$}
     \put(1.6,-0.2){$\circ$}
     \put(2.0,-0.2){$\bullet$}
     \end{picture}}
 
\newsavebox{\boxcrosspartwbbw}
   \savebox{\boxcrosspartwbbw}
   { \begin{picture}(1,1.4)
     \put(0.25,0.15){\line(1,2){0.4}}
     \put(0.65,0.15){\line(-1,2){0.4}}
     \put(0,0.9){$\circ$}
     \put(0.5,0.9){$\bullet$}
     \put(0,-0.2){$\bullet$}
     \put(0.5,-0.2){$\circ$}
     \end{picture}}
\newsavebox{\boxcrosspartbwwb}
   \savebox{\boxcrosspartbwwb}
   { \begin{picture}(1,1.4)
     \put(0.25,0.15){\line(1,2){0.4}}
     \put(0.65,0.15){\line(-1,2){0.4}}
     \put(0,0.9){$\bullet$}
     \put(0.5,0.9){$\circ$}
     \put(0,-0.2){$\circ$}
     \put(0.5,-0.2){$\bullet$}
     \end{picture}}
\newsavebox{\boxcrosspartwwww}
   \savebox{\boxcrosspartwwww}
   { \begin{picture}(1,1.4)
     \put(0.25,0.15){\line(1,2){0.4}}
     \put(0.65,0.15){\line(-1,2){0.4}}
     \put(0,0.9){$\circ$}
     \put(0.5,0.9){$\circ$}
     \put(0,-0.2){$\circ$}
     \put(0.5,-0.2){$\circ$}
     \end{picture}}
\newsavebox{\boxcrosspartbbbb}
   \savebox{\boxcrosspartbbbb}
   { \begin{picture}(1,1.4)
     \put(0.25,0.15){\line(1,2){0.4}}
     \put(0.65,0.15){\line(-1,2){0.4}}
     \put(0,0.9){$\bullet$}
     \put(0.5,0.9){$\bullet$}
     \put(0,-0.2){$\bullet$}
     \put(0.5,-0.2){$\bullet$}
     \end{picture}}
\newsavebox{\boxcrosspartonelinewwww}
   \savebox{\boxcrosspartonelinewwww}
   { \begin{picture}(1.7,1)
     \put(0.2,0.2){\line(0,1){0.4}}
     \put(0.6,0.2){\line(0,1){0.6}}
     \put(1,0.2){\line(0,1){0.4}}
     \put(1.4,0.2){\line(0,1){0.6}}
     \put(0.2,0.6){\line(1,0){0.8}}
     \put(0.6,0.8){\line(1,0){0.8}}     
     \put(0,-0.2){$\circ$}
     \put(0.4,-0.2){$\circ$}
     \put(0.8,-0.2){$\circ$}
     \put(1.2,-0.2){$\circ$}
     \end{picture}}     

\newsavebox{\boxhalflibpartwwwwww}
   \savebox{\boxhalflibpartwwwwww}
   { \begin{picture}(1.5,1.4)
     \put(0.3,0.15){\line(1,1){0.8}}
     \put(1.1,0.15){\line(-1,1){0.8}}
     \put(0.7,0.15){\line(0,1){0.8}}     
     \put(0,0.9){$\circ$}
     \put(0.5,0.9){$\circ$}
     \put(1,0.9){$\circ$}
     \put(0,-0.2){$\circ$}
     \put(0.5,-0.2){$\circ$}
     \put(1,-0.2){$\circ$}     
     \end{picture}}

\newsavebox{\boxpositionerd} 
   \savebox{\boxpositionerd}
   { \begin{picture}(2.6,1.5)
     \put(0.9,0.2){\line(0,1){0.9}}
     \put(2.3,0.2){\line(0,1){0.9}}
     \put(0.9,1.1){\line(1,0){1.4}}
     \put(0.05,-0.05){$^\uparrow$}
     \put(0.2,0.4){\tiny$^{\otimes d}$}
     \put(1.35,-0.05){$^\uparrow$}
     \put(1.5,0.4){\tiny$^{\otimes d}$}
     \put(0,-0.2){$\circ$}
     \put(0.7,-0.2){$\circ$}
     \put(1.3,-0.2){$\bullet$}
     \put(2.1,-0.2){$\bullet$}
     \end{picture}}
\newsavebox{\boxpositionerrpluseins} 
   \savebox{\boxpositionerrpluseins}
   { \begin{picture}(3.9,1.5)
     \put(1.6,0.2){\line(0,1){0.9}}
     \put(3.6,0.2){\line(0,1){0.9}}
     \put(1.6,1.1){\line(1,0){2}}
     \put(0.05,-0.05){$^\uparrow$}
     \put(0.2,0.4){\tiny$^{\otimes r+1}$}
     \put(2.05,-0.05){$^\uparrow$}
     \put(2.2,0.4){\tiny$^{\otimes r-1}$}
     \put(0,-0.2){$\circ$}
     \put(1.4,-0.2){$\bullet$}
     \put(2,-0.2){$\bullet$}
     \put(3.4,-0.2){$\bullet$}
     \end{picture}}
\newsavebox{\boxpositionerdt} 
   \savebox{\boxpositionerdt}
   { \begin{picture}(2.9,1.5)
     \put(1.2,0.2){\line(0,1){0.9}}
     \put(2.9,0.2){\line(0,1){0.9}}
     \put(1.2,1.1){\line(1,0){1.7}}
     \put(0.05,-0.05){$^\uparrow$}
     \put(0.2,0.4){\tiny$^{\otimes dt}$}
     \put(1.65,-0.05){$^\uparrow$}
     \put(1.8,0.4){\tiny$^{\otimes dt}$}
     \put(0,-0.2){$\circ$}
     \put(1,-0.2){$\circ$}
     \put(1.6,-0.2){$\bullet$}
     \put(2.7,-0.2){$\bullet$}
     \end{picture}}
\newsavebox{\boxpositioners} 
   \savebox{\boxpositioners}
   { \begin{picture}(2.6,1.5)
     \put(0.9,0.2){\line(0,1){0.9}}
     \put(2.3,0.2){\line(0,1){0.9}}
     \put(0.9,1.1){\line(1,0){1.4}}
     \put(0.05,-0.05){$^\uparrow$}
     \put(0.2,0.4){\tiny$^{\otimes s}$}
     \put(1.35,-0.05){$^\uparrow$}
     \put(1.5,0.4){\tiny$^{\otimes s}$}
     \put(0,-0.2){$\circ$}
     \put(0.7,-0.2){$\circ$}
     \put(1.3,-0.2){$\bullet$}
     \put(2.1,-0.2){$\bullet$}
     \end{picture}}
\newsavebox{\boxpositionerdinv} 
   \savebox{\boxpositionerdinv}
   { \begin{picture}(2.6,1.5)
     \put(0.9,0.2){\line(0,1){0.9}}
     \put(2.3,0.2){\line(0,1){0.9}}
     \put(0.9,1.1){\line(1,0){1.4}}
     \put(0.05,-0.05){$^\uparrow$}
     \put(0.2,0.4){\tiny$^{\otimes d}$}
     \put(1.35,-0.05){$^\uparrow$}
     \put(1.5,0.4){\tiny$^{\otimes d}$}
     \put(0,-0.2){$\circ$}
     \put(0.7,-0.2){$\bullet$}
     \put(1.3,-0.2){$\bullet$}
     \put(2.1,-0.2){$\circ$}
     \end{picture}}
\newsavebox{\boxpositionersinv} 
   \savebox{\boxpositionersinv}
   { \begin{picture}(2.6,1.5)
     \put(0.9,0.2){\line(0,1){0.9}}
     \put(2.3,0.2){\line(0,1){0.9}}
     \put(0.9,1.1){\line(1,0){1.4}}
     \put(0.05,-0.05){$^\uparrow$}
     \put(0.2,0.4){\tiny$^{\otimes s}$}
     \put(1.35,-0.05){$^\uparrow$}
     \put(1.5,0.4){\tiny$^{\otimes s}$}
     \put(0,-0.2){$\circ$}
     \put(0.7,-0.2){$\bullet$}
     \put(1.3,-0.2){$\bullet$}
     \put(2.1,-0.2){$\circ$}
     \end{picture}}
\newsavebox{\boxpositionerdpluszwei}
   \savebox{\boxpositionerdpluszwei}
   { \begin{picture}(3.4,1.5)
     \put(1.6,0.2){\line(0,1){0.9}}
     \put(3.0,0.2){\line(0,1){0.9}}
     \put(1.6,1.1){\line(1,0){1.4}}
     \put(0.05,-0.05){$^\uparrow$}
     \put(0.2,0.4){\tiny$^{\otimes d+2}$}
     \put(2.05,-0.05){$^\uparrow$}
     \put(2.2,0.4){\tiny$^{\otimes d}$}
     \put(0,-0.2){$\circ$}
     \put(1.4,-0.2){$\bullet$}
     \put(2,-0.2){$\bullet$}
     \put(2.8,-0.2){$\bullet$}
     \end{picture}}
\newsavebox{\boxpositionersminuszwei}
   \savebox{\boxpositionersminuszwei}
   { \begin{picture}(3.4,1.5)
     \put(1,0.2){\line(0,1){0.9}}
     \put(3.0,0.2){\line(0,1){0.9}}
     \put(1,1.1){\line(1,0){2}}
     \put(0.05,-0.05){$^\uparrow$}
     \put(0.2,0.4){\tiny$^{\otimes s}$}
     \put(1.45,-0.05){$^\uparrow$}
     \put(1.6,0.4){\tiny$^{\otimes s-2}$}
     \put(0,-0.2){$\circ$}
     \put(0.8,-0.2){$\bullet$}
     \put(1.4,-0.2){$\bullet$}
     \put(2.8,-0.2){$\bullet$}
     \end{picture}}
\newsavebox{\boxpositionerrnull}
   \savebox{\boxpositionerrnull}
   { \begin{picture}(3.8,1.5)
     \put(1.2,0.2){\line(0,1){0.9}}
     \put(3.4,0.2){\line(0,1){0.9}}
     \put(1.2,1.1){\line(1,0){2.2}}
     \put(0.05,-0.05){$^\uparrow$}
     \put(0.2,0.4){\tiny$^{\otimes r_0}$}
     \put(1.65,-0.05){$^\uparrow$}
     \put(1.8,0.4){\tiny$^{\otimes r_0-2}$}
     \put(0,-0.2){$\circ$}
     \put(1,-0.2){$\bullet$}
     \put(1.6,-0.2){$\bullet$}
     \put(3.2,-0.2){$\bullet$}
     \end{picture}}
\newsavebox{\boxpositionerwbwb}
   \savebox{\boxpositionerwbwb}
   { \begin{picture}(1.8,1)
     \put(0.6,0.2){\line(0,1){0.6}}
     \put(1.4,0.2){\line(0,1){0.6}}
     \put(0.6,0.8){\line(1,0){0.8}}
     \put(0.05,-0.05){$^\uparrow$}
     \put(0.85,-0.05){$^\uparrow$}
     \put(0,-0.2){$\circ$}
     \put(0.4,-0.2){$\bullet$}
     \put(0.8,-0.2){$\circ$}
     \put(1.2,-0.2){$\bullet$}
     \end{picture}}
\newsavebox{\boxpositionerwwbb}
   \savebox{\boxpositionerwwbb}
   { \begin{picture}(1.8,1)
     \put(0.6,0.2){\line(0,1){0.6}}
     \put(1.4,0.2){\line(0,1){0.6}}
     \put(0.6,0.8){\line(1,0){0.8}}
     \put(0.05,-0.05){$^\uparrow$}
     \put(0.85,-0.05){$^\uparrow$}
     \put(0,-0.2){$\circ$}
     \put(0.4,-0.2){$\circ$}
     \put(0.8,-0.2){$\bullet$}
     \put(1.2,-0.2){$\bullet$}
     \end{picture}}
\newsavebox{\boxpositionerwwww}
   \savebox{\boxpositionerwwww}
   { \begin{picture}(1.8,1)
     \put(0.6,0.2){\line(0,1){0.6}}
     \put(1.4,0.2){\line(0,1){0.6}}
     \put(0.6,0.8){\line(1,0){0.8}}
     \put(0.05,-0.05){$^\uparrow$}
     \put(0.85,-0.05){$^\uparrow$}
     \put(0,-0.2){$\circ$}
     \put(0.4,-0.2){$\circ$}
     \put(0.8,-0.2){$\circ$}
     \put(1.2,-0.2){$\circ$}
     \end{picture}}
\newsavebox{\boxpositionerrevwbwb}
   \savebox{\boxpositionerrevwbwb}
   { \begin{picture}(1.8,1)
     \put(0.2,0.2){\line(0,1){0.6}}
     \put(1.0,0.2){\line(0,1){0.6}}
     \put(0.2,0.8){\line(1,0){0.8}}
     \put(0.45,-0.05){$^\uparrow$}
     \put(1.25,-0.05){$^\uparrow$}
     \put(0,-0.2){$\circ$}
     \put(0.4,-0.2){$\bullet$}
     \put(0.8,-0.2){$\circ$}
     \put(1.2,-0.2){$\bullet$}
     \end{picture}}
\newsavebox{\boxpositionersalphaw} 
   \savebox{\boxpositionersalphaw}
   { \begin{picture}(2.6,1.5)
     \put(0.9,0.2){\line(0,1){0.9}}
     \put(2.3,0.2){\line(0,1){0.9}}
     \put(0.9,1.1){\line(1,0){1.4}}
     \put(0.05,-0.05){$^\uparrow$}
     \put(1.35,-0.05){$^\uparrow$}
     \put(0.2,0.4){\tiny$^{\otimes s}$}
     \put(1.5,0.4){\tiny$^{\otimes \alpha}$}
     \put(0,-0.2){$\circ$}
     \put(0.7,-0.2){$\circ$}
     \put(1.3,-0.2){$\circ$}
     \put(2.1,-0.2){$\bullet$}
     \end{picture}}
\newsavebox{\boxpositionersalphab} 
   \savebox{\boxpositionersalphab}
   { \begin{picture}(2.6,1.5)
     \put(0.9,0.2){\line(0,1){0.9}}
     \put(2.3,0.2){\line(0,1){0.9}}
     \put(0.9,1.1){\line(1,0){1.4}}
     \put(0.05,-0.05){$^\uparrow$}
     \put(1.35,-0.05){$^\uparrow$}
     \put(0.2,0.4){\tiny$^{\otimes s}$}
     \put(1.5,0.4){\tiny$^{\otimes \alpha}$}
     \put(0,-0.2){$\circ$}
     \put(0.7,-0.2){$\circ$}
     \put(1.3,-0.2){$\bullet$}
     \put(2.1,-0.2){$\bullet$}
     \end{picture}}

\newsavebox{\boxfatcrosswwww}
   \savebox{\boxfatcrosswwww}
   { \begin{picture}(2,1.6)
     \put(0.2,0.15){\line(0,1){0.2}}
     \put(0.2,0.95){\line(0,1){0.2}}
     \put(0.2,0.95){\line(1,0){0.5}}
     \put(0.2,0.35){\line(1,0){0.5}}
     \put(0.7,0.15){\line(0,1){0.2}}
     \put(0.7,0.95){\line(0,1){0.2}}
     \put(0,1.1){$\circ$}
     \put(0.5,1.1){$\circ$}
     \put(0,-0.2){$\circ$}
     \put(0.5,-0.2){$\circ$}
      \put(0.5,0.35){\line(2,1){1}}
      \put(0.5,0.85){\line(2,-1){1}}
     \put(1.2,0.15){\line(0,1){0.2}}
     \put(1.2,0.95){\line(0,1){0.2}}
     \put(1.2,0.95){\line(1,0){0.5}}
     \put(1.2,0.35){\line(1,0){0.5}}
     \put(1.7,0.15){\line(0,1){0.2}}
     \put(1.7,0.95){\line(0,1){0.2}}
     \put(1,1.1){$\circ$}
     \put(1.5,1.1){$\circ$}
     \put(1,-0.2){$\circ$}
     \put(1.5,-0.2){$\circ$}
     \end{picture}}
\newsavebox{\boxpairpositionerwwwwww}
   \savebox{\boxpairpositionerwwwwww}
   { \begin{picture}(2,1.6)
     \put(0.2,0.95){\line(0,1){0.2}}
     \put(0.2,0.95){\line(1,0){0.5}}
     \put(0.7,0.95){\line(0,1){0.2}}
     \put(0,1.1){$\circ$}
     \put(0.5,1.1){$\circ$}
     \put(0.25,-0.2){$\circ$}
       \put(0.45,0.1){\line(1,1){1}}
      \put(0.5,0.85){\line(2,-1){1}}
     \put(1.2,0.15){\line(0,1){0.2}}
     \put(1.2,0.35){\line(1,0){0.5}}
     \put(1.7,0.15){\line(0,1){0.2}}
     \put(1.25,1.1){$\circ$}
     \put(1,-0.2){$\circ$}
     \put(1.5,-0.2){$\circ$}
     \end{picture}}     

\newsavebox{\boxAspace}
   \savebox{\boxAspace}
   { \begin{picture}(0.1,1.5)
     \end{picture}}
\newsavebox{\boxBspace}
   \savebox{\boxBspace}
   { \begin{picture}(0,1.85)
     \end{picture}}

\newcommand{\idpartww}{\usebox{\boxidpartww}}
\newcommand{\idpartwb}{\usebox{\boxidpartwb}}
\newcommand{\idpartbw}{\usebox{\boxidpartbw}}
\newcommand{\idpartbb}{\usebox{\boxidpartbb}}
\newcommand{\idpartsingletonww}{\usebox{\boxidpartsingletonww}}
\newcommand{\idpartsingletonbb}{\usebox{\boxidpartsingletonbb}}
\newcommand{\paarpartww}{\usebox{\boxpaarpartww}}
\newcommand{\paarpartbw}{\usebox{\boxpaarpartbw}}
\newcommand{\paarpartwb}{\usebox{\boxpaarpartwb}}
\newcommand{\paarpartbb}{\usebox{\boxpaarpartbb}}
\newcommand{\baarpartww}{\usebox{\boxbaarpartww}}
\newcommand{\baarpartbw}{\usebox{\boxbaarpartbw}}
\newcommand{\baarpartwb}{\usebox{\boxbaarpartwb}}

\newcommand{\paarbaarpartwwww}{\usebox{\boxpaarbaarpartwwww}}
\newcommand{\singletonw}{\usebox{\boxsingletonw}}
\newcommand{\singletonb}{\usebox{\boxsingletonb}}
\newcommand{\downsingletonw}{\usebox{\boxdownsingletonw}}

\newcommand{\vierpartwbwb}{\usebox{\boxvierpartwbwb}}

\newcommand{\vierpartwwbb}{\usebox{\boxvierpartwwbb}}

\newcommand{\vierpartrotwbwb}{\usebox{\boxvierpartrotwbwb}}

\newcommand{\vierpartrotbwbw}{\usebox{\boxvierpartrotbwbw}}
\newcommand{\vierpartrotwwww}{\usebox{\boxvierpartrotwwww}}
\newcommand{\vierpartrotbbbb}{\usebox{\boxvierpartrotbbbb}}

\newcommand{\dreipartwww}{\usebox{\boxdreipartwww}}

\newcommand{\dreipartrotwww}{\usebox{\boxdreipartrotwww}}
\newcommand{\downdreipartrotwww}{\usebox{\boxdowndreipartrotwww}}

\newcommand{\crosspartwbbw}{\usebox{\boxcrosspartwbbw}}
\newcommand{\crosspartbwwb}{\usebox{\boxcrosspartbwwb}}
\newcommand{\crosspartwwww}{\usebox{\boxcrosspartwwww}}
\newcommand{\crosspartbbbb}{\usebox{\boxcrosspartbbbb}}
\newcommand{\crosspartonelinewwww}{\usebox{\boxcrosspartonelinewwww}}
\newcommand{\halflibpartwwwwww}{\usebox{\boxhalflibpartwwwwww}}
\newcommand{\positionerd}{\usebox{\boxpositionerd}}
\newcommand{\positionerrpluseins}{\usebox{\boxpositionerrpluseins}}

\newcommand{\positionerwbwb}{\usebox{\boxpositionerwbwb}}

\newcommand{\positionerwwww}{\usebox{\boxpositionerwwww}}

\newcommand{\paarpart}{\paarpartww}
\newcommand{\baarpart}{\baarpartww}

\newcommand{\dreipartrot}{\dreipartrotwww}

\newcommand{\downdreipartrot}{\downdreipartrotwww}

\newcommand{\vierpartrot}{\vierpartrotwwww}
\newcommand{\crosspart}{\crosspartwwww}
\newcommand{\crosspartoneline}{\crosspartonelinewwww}
\newcommand{\halflibpart}{\halflibpartwwwwww}
\newcommand{\fatcross}{\usebox{\boxfatcrosswwww}}

\newcommand{\idpart}{\idpartww}
\newcommand{\singleton}{\singletonw}
\newcommand{\downsingleton}{\downsingletonw}
\newcommand{\legpart}{\usebox{\boxpairpositionerwwwwww}}
\newcommand{\positioner}{\positionerwwww}

\theoremstyle{plain} 
\newtheorem{thm}{Theorem}[section]

\newtheorem{lem}[thm]{Lemma}
\newtheorem{cor}[thm]{Corollary}
\theoremstyle{definition}
\newtheorem{defn}[thm]{Definition}

\newtheorem{ex}[thm]{Example}
\newtheorem{quest}[thm]{Question}
\newtheorem{prob}[thm]{Problem}

\numberwithin{equation}{section}

\renewcommand{\theta}{\vartheta}
\renewcommand{\phi}{\varphi}
\renewcommand{\epsilon}{\varepsilon}
\renewcommand{\subset}{\subseteq}

\newcommand{\N}{\mathbb N}
\newcommand{\Z}{\mathbb Z}

\newcommand{\C}{\mathbb C}

\newcommand{\CC}{\mathcal C}

\newcommand{\twocol}{\circ\bullet}
\DeclareMathOperator{\rl}{rl}
\DeclareMathOperator{\rot}{rot}

\begin{document}
\title{Partition $C^*$-algebras}
\author{Moritz Weber}
\address{Saarland University, Fachbereich Mathematik, Postfach 151150,
66041 Saarbr\"ucken, Germany}
\email{weber@math.uni-sb.de}
\date{\today}
\thanks{The author was supported by the ERC Advanced Grant NCDFP, held by Roland Speicher, by the SFB-TRR 195, and by the DFG project \emph{Quantenautomorphismen von Graphen}.}

\begin{abstract}
We give a definition of partition $C^*$-algebras: To any  partition of a finite set, we assign algebraic relations for a matrix of generators of a universal $C^*$-algebra. We then prove how certain relations may be deduced from others and we explain a partition calculus for simplifying such computations. This article is a small note for $C^*$-algebraists having no background in compact quantum groups, although our partition $C^*$-algebras are motivated from those underlying  Banica-Speicher quantum groups (also called easy quantum groups). We list many open questions about partition $C^*$-algebras that may be tackled by purely $C^*$-algebraic means, ranging from ideal structures and representations on Hilbert spaces to $K$-theory and isomorphism questions.
In a follow up article, we deal with the quantum algebraic structure associated to partition $C^*$-algebras.
\end{abstract}

\maketitle

\section*{Introduction}

The theory of Banica-Speicher quantum groups (also called easy quantum groups) \cite{BS,WeEQGLN,WeLInd} is a quite combinatorial approach to Woronowicz's compact matrix quantum groups \cite{WoCMQG,Tim,NT}. The strategy is to build suitable tensor categories out of ensembles of set partitions, and then to view them as intertwiner spaces of some compact matrix quantum groups, due to Woronowicz's Tannaka-Krein theory \cite{WoTK}, a kind of quantum Schur-Weyl theory. However, these compact matrix quantum groups may simply be viewed as $C^*$-algebras endowed with some additional quantum algebraic structure -- and the link from set partitions to those $C^*$-algebras may be given directly, in a purely $C^*$-algebraic way.

\textbf{The purpose of this note is to extract the purely $C^*$-algebraic essence} of the above Banica-Speicher theor, to extend it slightly, and to formulate a class of $C^*$-algebras based on set partitions:
\begin{displaymath}
\xymatrix{
  \textnormal{partitions of finite sets} \ar[r]
 &\textnormal{$C^*$-algebraic relations} \ar[r]
 &\textnormal{partition $C^*$-algebras}
 }
\end{displaymath}

The general philosophy is, that much of the structure of a partition $C^*$-algebra is inherent in its combinatorial data. In this sense, partition $C^*$-algebras form a quite combinatorial class of $C^*$-algebras.

In order to get an impression of these relations coming from partitions, let $p$ be a partition of the ordered set $\{1,\ldots,k+l\}$ (i.e. a decomposition into a union of disjoint subsets) and let $n\in\N$. We consider the following relations $R(p)$ for self-adjoint elements $u_{ij}$, with $i,j=1,\ldots,n$:
\[\boxed{R(p):\qquad\sum_{\gamma_1,\ldots,\gamma_k=1}^n \delta_p(\gamma,\beta) u_{\gamma_1\alpha_1}\ldots u_{\gamma_k\alpha_k}
=\sum_{\gamma_1',\ldots,\gamma
_l'=1}^n \delta_p(\alpha,\gamma') u_{\beta_1\gamma_1'}\ldots u_{\beta_l\gamma_l'}\qquad}\]
Here, $\alpha=(\alpha_1,\ldots,\alpha_k)$ and $\beta=(\beta_1,\ldots,\beta_l)$ are multi indices over the set $\{1,\ldots,n\}$ and $\delta_p(\alpha,\beta)$ is either one or zero, depending on whether or not the indices $\alpha$ and $\beta$ match the partition structure (see Definition \ref{DefRP} for a precise definition). If $X$ is a collection of partitions,  a partition $C^*$-algebra is by definition the following unital universal $C^*$-algebra:
\[\boxed{A_n(X):=C^*(1,u_{ij}, i,j=1,\ldots,n\;|\; u_{ij}=u_{ij}^*, \textnormal{the relations $R(p)$ hold for all }p\in X)}\]
We explain a partition calculus suitable for deriving how relations $R(p)$ imply other relations $R(q)$. As an example, using this partition calculus, it is very easy to see that the relations 
\[\delta_{\beta_1\beta_4}\delta_{\beta_2\beta_3}\delta_{\beta_3\beta_5}\delta_{\beta_6\beta_8}
u_{\beta_2\alpha_1}u_{\beta_7\alpha_2}
=\sum_{\gamma'_1,\gamma'_2}u_{\beta_1\gamma'_1}u_{\beta_2\alpha_1}u_{\beta_3\alpha_1}u_{\beta_4\gamma'_1}u_{\beta_5\alpha_1}u_{\beta_6\gamma'_2}
u_{\beta_7\alpha_2}u_{\beta_8\gamma'_2}\]
and
\[\sum_\gamma u_{\gamma \alpha_1}u_{\beta_2 \alpha_2}u_{\gamma \alpha_3}u_{\beta_1 \alpha_4}u_{\beta_2 \alpha_5}=\delta_{\alpha_1\alpha_3}\delta_{\alpha_2\alpha_5}u_{\beta_1\alpha_4}u_{\beta_2\alpha_2}\]
imply that $A_n(X)$ is commutative. This partition calculus allows us to conclude that \textbf{partition $C^*$-algebras are of combinatorial nature}.\\

Although endowing  $A_n(X)$ with some Hopf algebraic structure yields quantum groups which are well-studied, many very basic questions about the $C^*$-algebraic properties of $A_n(X)$ still remain mysterious. Amongst others, the ideal structure of $A_n(X)$ is far from being understood; in particular almost nothing is known about the kernel of maps
\[A_n(X)\to A_n(Y)\]
with $X\subset Y$. Moreover, the $K$-theory of $A_n(X)$ is known only in very special cases. Finally, we do not have enough concrete representations of $A_n(X)$ as operators on a Hilbert space. Therefore, amongst others, quite a number of questions around isomorphisms of partition $C^*$-algebras remain open. It would be very fruitful for the theory of Banica-Speicher quantum groups and in general, for compact quantum groups, if some of the questions could be answered. 

In Section \ref{SectOpen}, we provide a list of open questions concerning partition $C^*$-algebras, formulated within the theory of $C^*$-algebras. Our partition $C^*$-algebras are \emph{not} $C^*$-algebraic completions of the partition algebras in \cite{PartAlg}.

\pagebreak

\section{The combinatorial data: set partitions}

\subsection{Partitions}

Let $k,l\in\N_0$ and consider the finite, ordered set $\{1,\ldots,k+l\}$. A decomposition of this set into disjoint, non-empty subsets $V_1,\ldots,V_m$ (called the \emph{blocks}) whose union is the complete set, is called a \emph{(set) partition}. We represent such partitions by pictures, drawing $k$ points on an upper line and $l$ points on a lower line with the convention of numbering the points counterclockwise, starting with the lower left point; we then connect those points belonging to the same block by strings (the strings are drawn inside a rectangle spanned by the upper and the lower line of points). As an example with $k=4$ and $l=5$, the partition
\[V_1=\{1,7,9\},\quad V_2=\{2,5\},\quad V_3=\{3\},\quad V_4=\{4\},\quad V_5=\{6,8\}\]
of the set $\{1,\ldots,9\}$ may be drawn as:
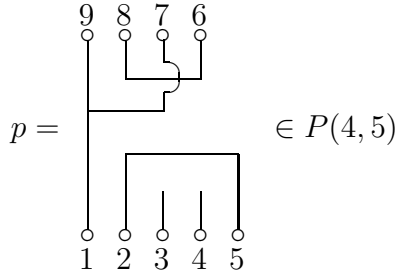
\newsavebox{\boxp}
   \savebox{\boxp}
   { \begin{picture}(4,5.5)
     \put(-1,6.35){\parti{5}{1}}
     \put(-1,6.35){\partii{1}{2}{4}}
     \put(0.3,3.5){\line(1,0){2}}
     \put(2.3,3.5){\line(0,1){0.5}}
     \put(2.3,4.8){\line(0,1){0.6}}
     \put(2.3,4.4){\oval(0.8,0.8)[r]}
     \put(0.05,5.3){$\circ$}
     \put(1.05,5.3){$\circ$}
     \put(2.05,5.3){$\circ$}
     \put(3.05,5.3){$\circ$}
     \put(-1,0.35){\uppartii{2}{2}{5}}
     \put(-1,0.35){\upparti{1}{3}}
     \put(-1,0.35){\upparti{1}{4}}     
     \put(0.05,0){$\circ$}
     \put(1.05,0){$\circ$}
     \put(2.05,0){$\circ$}
     \put(3.05,0){$\circ$}     
     \put(4.05,0){$\circ$}  
     \end{picture}}     
\begin{center}
\begin{picture}(10,7.5)
 \put(0,3.5){$p=$}
 \put(1.5,0.7){\usebox{\boxp}}
 \put(7,3.5){$\in P(4,5)$}
 \put(1.8,0){1}
 \put(2.8,0){2}
 \put(3.8,0){3}
 \put(4.8,0){4}
 \put(5.8,0){5}
 \put(1.8,6.5){9}
 \put(2.8,6.5){8}
 \put(3.8,6.5){7}
 \put(4.8,6.5){6}  
\end{picture}
\end{center}
We usually omit to write down the numberings of the points. The set of all partitions with $k$ upper and $l$ lower points is denoted by $P(k,l)$, and the collection of all $P(k,l)$ is denoted by $P$. Note that $k=0$ and $l=0$ are allowed.
Some basic examples of partitions are the following ones, each consisting of a single block apart from the crossing partition (which consists in two blocks):
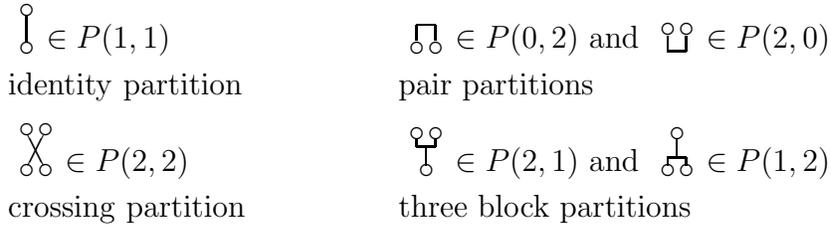
\begin{align*}
&\idpart\in P(1,1) &&\paarpart\in P(0,2) \textnormal{ and }\baarpart\in P(2,0)\\
&\textnormal{identity partition} &&\textnormal{pair partitions}\\
&\crosspart\in P(2,2) &&\dreipartrot\in P(2,1)\textnormal{ and }\downdreipartrot\in P(1,2)\\
&\textnormal{crossing partition} &&\textnormal{three block partitions}
\end{align*}
See also Appendix B for more examples, and \cite{BS,VSW}.

\subsection{Operations on partitions}
\label{SectOper}

We have several operations on the sets $P$ and $P\times P$. Let $p\in P(k,l)$ and $q\in P(k',l')$ be given and let the partitions
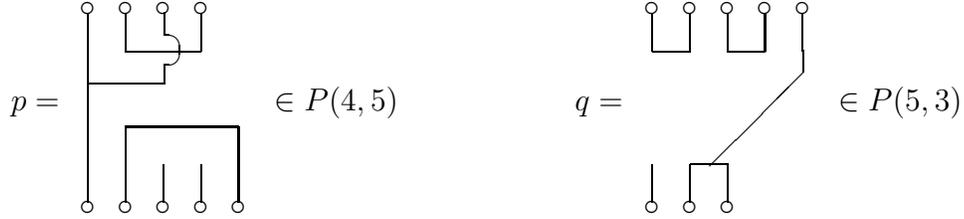
\newsavebox{\boxq}
   \savebox{\boxq}
   { \begin{picture}(4,5.5)
     \put(-1,6.35){\partii{1}{1}{2}}
     \put(-1,6.35){\partii{1}{3}{4}}
     \put(-1,6.35){\parti{1}{5}}
     \put(1.8,1.3){\line(1,1){2.5}}
     \put(4.3,3.8){\line(0,1){0.6}}
     \put(0.05,5.3){$\circ$}
     \put(1.05,5.3){$\circ$}
     \put(2.05,5.3){$\circ$}
     \put(3.05,5.3){$\circ$}
     \put(4.05,5.3){$\circ$}
     \put(-1,0.35){\upparti{1}{1}}
     \put(-1,0.35){\uppartii{1}{2}{3}}
     \put(0.05,0){$\circ$}
     \put(1.05,0){$\circ$}
     \put(2.05,0){$\circ$}
     \end{picture}}     
\newsavebox{\boxqshort}
   \savebox{\boxqshort}
   { \begin{picture}(4,5.5)
     \put(-1,4.35){\partii{1}{1}{2}}
     \put(-1,4.35){\partii{1}{3}{4}}
     \put(-1,4.35){\parti{1}{5}}
     \put(1.8,1.3){\line(3,1){2.5}}
     \put(4.3,2.2){\line(0,1){0.2}}
     \put(0.05,3.3){$\circ$}
     \put(1.05,3.3){$\circ$}
     \put(2.05,3.3){$\circ$}
     \put(3.05,3.3){$\circ$}
     \put(4.05,3.3){$\circ$}
     \put(-1,0.35){\upparti{1}{1}}
     \put(-1,0.35){\uppartii{1}{2}{3}}
     \put(0.05,0){$\circ$}
     \put(1.05,0){$\circ$}
     \put(2.05,0){$\circ$}
     \end{picture}}     
\begin{center}
\begin{picture}(25,7.5)
 \put(0,3.5){$p=$}
 \put(1.5,0.7){\usebox{\boxp}}
 \put(7,3.5){$\in P(4,5)$}
 \put(15,3.5){$q=$}
 \put(16.5,0.7){\usebox{\boxq}}
 \put(22,3.5){$\in P(5,3)$}
\end{picture}
\end{center}
be our guiding sample partitions. The \emph{tensor product} of $p$ and $q$ is the partition $p\otimes q\in P(k+k',l+l')$ obtained by horizontal concatenation, i.e. writing $p$ and $q$ side by side:
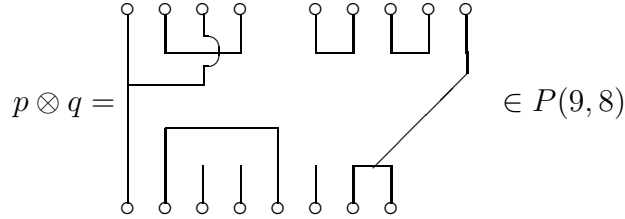
\begin{center}
\begin{picture}(16,7.5)
 \put(0,3.5){$p\otimes q=$}
 \put(2.5,0.7){\usebox{\boxp}}
 \put(7.5,0.7){\usebox{\boxq}}
 \put(13,3.5){$\in P(9,8)$}
\end{picture}
\end{center}
If $l=k'$, then the \emph{composition} of $q$ and $p$ is the partition $qp\in P(k,l')$ obtained by vertical concatenation, i.e. writing $q$ below $p$. 
First connect $k$ upper points by $p$ to $l=k'$ middle points and then connect these points  by $q$ to $l'$ lower points. This yields a partition, connecting $k$ upper points with $l'$ lower points. The $l$ middle points are removed. By the composition procedure, certain composed strings (called \emph{loops}) may appear, which are neither connected to upper nor to lower points; they result from blocks around the middle points. We remove these loops and denote their number by $\rl(q,p)$.
\newsavebox{\boxqphere}
   \savebox{\boxqphere}
   
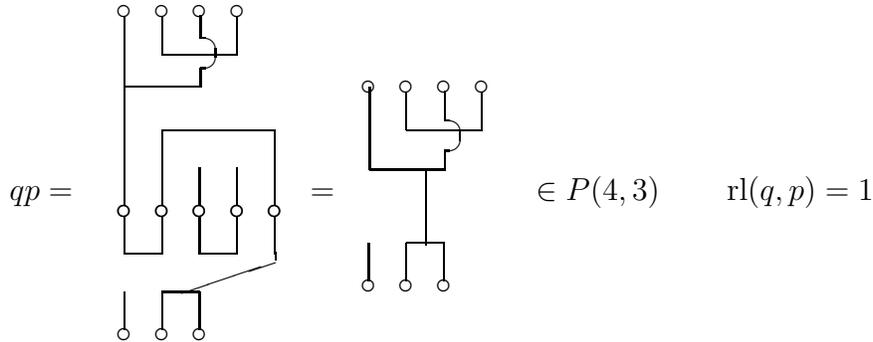
{ \begin{picture}(4,5.5)
     \put(-1,6.35){\partii{1}{2}{4}}
     \put(0.3,3.3){\line(0,1){2.2}}
     \put(0.3,3.3){\line(1,0){2}}
     \put(2.3,3.3){\line(0,1){0.5}}
     \put(2.3,4.6){\line(0,1){0.8}}
     \put(2.3,4.2){\oval(0.8,0.8)[r]}
     \put(1.8,1.3){\line(0,1){2}}     
     \put(0.05,5.3){$\circ$}
     \put(1.05,5.3){$\circ$}
     \put(2.05,5.3){$\circ$}
     \put(3.05,5.3){$\circ$}
     \put(-1,0.35){\upparti{1}{1}}
     \put(-1,0.35){\uppartii{1}{2}{3}}
     \put(0.05,0){$\circ$}
     \put(1.05,0){$\circ$}
     \put(2.05,0){$\circ$}
     \end{picture}} 
\begin{center}
\begin{picture}(23,10)
 \put(0,4.5){$qp=$}
 \put(2.5,4){\usebox{\boxp}}
 \put(2.5,0.7){\usebox{\boxqshort}}
 \put(8,4.5){$=$}
 \put(9,2){\usebox{\boxqphere}}
 \put(14,4.5){$\in P(4,3)\qquad \rl(q,p)=1$}
\end{picture}
\end{center}
The \emph{vertical reflection} of $p$ is the partition $\tilde p\in P(k,l)$ obtained by reflection at the vertical axis, whereas the \emph{involution} of $p$ is the partition $p^*\in P(l,k)$ obtained by reflection at the horizontal axis, i.e. turning $p$ upside down:
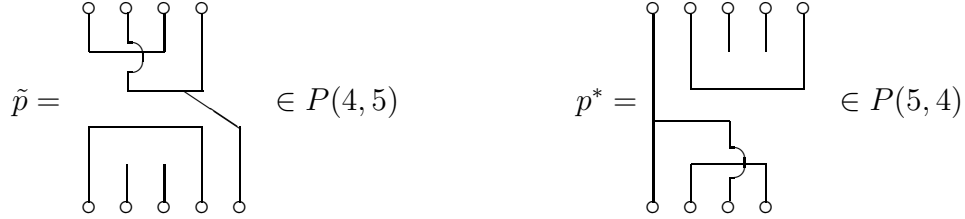
\newsavebox{\boxptildeh}
   \savebox{\boxptildeh}
   { \begin{picture}(4,5.5)
     \put(-1,6.35){\partii{1}{1}{3}}
     \put(-1,6.35){\parti{2}{4}}     
     \put(1.3,3.3){\line(1,0){2}}
     \put(1.3,3.3){\line(0,1){0.5}}
     \put(1.3,4.6){\line(0,1){0.8}}
     \put(1.3,4.2){\oval(0.8,0.8)[r]}
     \put(2.8,3.3){\line(3,-2){1.5}}
     \put(0.05,5.3){$\circ$}
     \put(1.05,5.3){$\circ$}
     \put(2.05,5.3){$\circ$}
     \put(3.05,5.3){$\circ$}
     \put(-1,0.35){\uppartii{2}{1}{4}}
     \put(-1,0.35){\upparti{1}{2}}
     \put(-1,0.35){\upparti{1}{3}}     
     \put(-1,0.35){\upparti{2}{5}}   
     \put(0.05,0){$\circ$}
     \put(1.05,0){$\circ$}
     \put(2.05,0){$\circ$}
     \put(3.05,0){$\circ$}     
     \put(4.05,0){$\circ$}  
     \end{picture}}     
\newsavebox{\boxpsternh}
   \savebox{\boxpsternh}
   { \begin{picture}(4,5.5)
     \put(-1,6.35){\parti{5}{1}}              
     \put(-1,6.35){\partii{2}{2}{5}}
     \put(-1,6.35){\parti{1}{3}}     
     \put(-1,6.35){\parti{1}{4}} 
     \put(0.3,2.5){\line(1,0){2}}
     \put(2.3,2.5){\line(0,-1){0.7}}
     \put(2.3,0.4){\line(0,1){0.6}}
     \put(2.3,1.4){\oval(0.8,0.8)[r]}
     \put(0.05,5.3){$\circ$}
     \put(1.05,5.3){$\circ$}
     \put(2.05,5.3){$\circ$}
     \put(3.05,5.3){$\circ$}
     \put(4.05,5.3){$\circ$}
     \put(-1,0.35){\uppartii{1}{2}{4}}
     \put(0.05,0){$\circ$}
     \put(1.05,0){$\circ$}
     \put(2.05,0){$\circ$}
     \put(3.05,0){$\circ$}     
     \end{picture}}          
\begin{center}
\begin{picture}(26,7.5)
 \put(0,3.5){$\tilde p=$}
 \put(1.5,0.7){\usebox{\boxptildeh}}
 \put(7,3.5){$\in P(4,5)$}
 \put(15,3.5){$p^*=$}
 \put(16.5,0.7){\usebox{\boxpsternh}}
 \put(22,3.5){$\in P(5,4)$}
\end{picture}
\end{center}
For more examples on the above operations, see Appendix B. See also \cite{BS} for a first definition of these operations, and see \cite{We,WeEQGLN,WeLInd} for more on it.

\subsection{Categories of partitions}
\label{SectCateg}

As a side remark, let us mention the following. A collection $\mathcal C$ of subsets $\mathcal C(k,l)\subseteq P(k,l)$ (for all $k,l\in\N_0$) is a \emph{category of partitions}, if it is closed under the above operations (tensor product, composition, involution and vertical reflection) and if it contains the pair partitions $\paarpart$ and $\baarpart$ as well as the identity partition $\idpart$. (Actually, for categories of partitions the vertical reflection may be deduced from the other operations, see \cite[Lem. 1.1]{TWcomb}.)
Categories of partitions play a major role in the theory of Banica-Speicher quantum groups \cite{BS}, see also Section \ref{SectCQG} for more on this subject.
Examples of categories of partitions include the set of all partitions $P$, the set of all \emph{pair partions} $P_2$ (each block of each partition consists of exactly two points), the set of all \emph{non-crossing partitions} $NC$ (each partition may be drawn in such a way that the strings do not cross), and the set of all \emph{non-crossing pair partitions} $NC_2$. See \cite{BS,We,WeEQGLN,WeLInd}.

\section{Partition $C^*$-algebras: The self-adjoint case}
\label{SectOnecolor}

We now define partition $C^*$-algebras in the case of self-adjoint generators $u_{ij}$. For more general cases, see Section \ref{SectGeneral}.

\subsection{The relations associated to partitions}

Let $p\in P(k,l)$ be a partition on $k$ upper and $l$ lower points, and let $n\in\N$. Let $i=(i_1,\ldots,i_k)$ and $j=(j_1,\ldots,j_l)$ be two multi indices with $i_s, j_t\in\{1,\ldots,n\}$ for all $1\leq s\leq k$ and all $1\leq t\leq l$. We also say that $i$ is a \emph{multi index of length $k$}, while $j$ is a multi index of length $l$.
We label the $k$ upper points of the partition $p$ with the indices $i_1,\ldots,i_k$ (from left to right) and respectively the $l$ lower points of $p$ with $j_1,\ldots,j_l$ (again from left to right).
We put:
\[\delta_p(i,j):=\begin{cases} 1 &\textnormal{if the strings of $p$ connect only equal indices}\\
0 &\textnormal{otherwise}\end{cases}\]
Note that if $k=0$, then $\delta_p(\emptyset,j)$ is still well-defined, likewise $\delta_p(i,\emptyset)$. As an example with $i=(2,4,2,4)$, $i'=(3,4,2,4)$ and $j=(2,5,5,7,5)$:
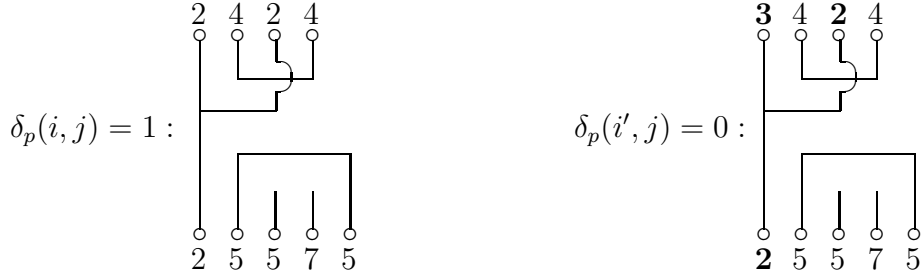
\begin{center}
\begin{picture}(24,7.5)
 \put(-1,3.5){$\delta_p(i,j)=1:$}
 \put(3.5,0.7){\usebox{\boxp}}
 \put(3.8,0){2}
 \put(4.8,0){5}
 \put(5.8,0){5}
 \put(6.8,0){7}
 \put(7.8,0){5}
 \put(3.8,6.5){2}
 \put(4.8,6.5){4}
 \put(5.8,6.5){2}
 \put(6.8,6.5){4} 
 \put(14,3.5){$\delta_p(i',j)=0:$}
 \put(18.5,0.7){\usebox{\boxp}}
 \put(18.8,0){\textbf 2}
 \put(19.8,0){5}
 \put(20.8,0){5}
 \put(21.8,0){7}
 \put(22.8,0){5}
 \put(18.8,6.5){\textbf 3}
 \put(19.8,6.5){4}
 \put(20.8,6.5){\textbf 2}
 \put(21.8,6.5){4}    
\end{picture}
\end{center}
\begin{defn}\label{DefRP}
Let $n\in\N$ and let $A$ be a unital $C^*$-algebra generated by $n^2$ elements $u_{ij}$, $1\leq i,j\leq n$.
Let $p\in P(k,l)$ be a partition. We say that the generators $u_{ij}$ \emph{fulfill the relations} $R(p)$, if
all elements $u_{ij}$ are self-adjoint, and for all $\alpha_1,\ldots, \alpha_k\in\{1,\ldots,n\}$ and for all $\beta_1,\ldots,\beta_l\in\{1,\ldots,n\}$, we have:
\[\boxed{\sum_{\gamma_1,\ldots,\gamma_k=1}^n \delta_p(\gamma,\beta) u_{\gamma_1\alpha_1}\ldots u_{\gamma_k\alpha_k}
=\sum_{\gamma_1',\ldots,\gamma_l'=1}^n \delta_p(\alpha,\gamma') u_{\beta_1\gamma_1'}\ldots u_{\beta_l\gamma_l'}}\]
For the cases $k=0$ or $l=0$, we use the convention:
\begin{align*}
\delta_p(\emptyset,\beta)1
&=\sum_{\gamma_1',\ldots,\gamma_l'=1}^n \delta_p(\emptyset,\gamma') u_{\beta_1\gamma_1'}\ldots u_{\beta_l\gamma_l'}
&\textnormal{if } k= 0, l\neq 0\\
\sum_{\gamma_1,\ldots,\gamma_k=1}^n \delta_p(\gamma,\emptyset) u_{\gamma_1\alpha_1}\ldots u_{\gamma_k\alpha_k}
&=\delta_p(\alpha,\emptyset)1
&\textnormal{if } k\neq 0, l= 0
\end{align*}

\end{defn}

\begin{ex}\label{ExRel}
For many partitions $p$, the relations $R(p)$ boil down to very simple relations. See also Appendix A for a list of more relations. Let $u:=(u_{ij})_{i,j=1,\ldots,n}$.
\begin{itemize}
\item[(a)] $R(\idpart)$: $u_{ij}=u_{ij}$ for all $i,j$ (i.e. no relation).
\item[(b)] $R(\paarpart)$: $\sum_m u_{im}u_{jm}=\delta_{ij}$ for all $i,j$, which means $uu^t=1$.
\item[(c)] $R(\baarpart)$: $\sum_m u_{mi}u_{mj}=\delta_{ij}$ for all $i,j$, which means $u^tu=1$.
\item[(d)] $R(\dreipartrot)$: $u_{mi}u_{mj}=\delta_{ij}u_{mi}$ for all $i,j,m$.
\item[(e)] $R(\downdreipartrot)$: $u_{im}u_{jm}=\delta_{ij}u_{im}$ for all $i,j,m$.
\item[(f)] $R(\crosspart)$: all $u_{ij}$ commute.
\end{itemize}
\end{ex}

\subsection{Definition of partition $C^*$-algebras}

\begin{defn}\label{DefCStar}
Let $X\subset P$ be a set of partitions. If the unital universal $C^*$-algebra
\[\boxed{A_n(X):=C^*(1, u_{ij}, i,j=1,\ldots,n\;|\; u_{ij}=u_{ij}^*, \textnormal{the relations $R(p)$ hold for all }p\in X)}\]
exists\footnote{i.e., the $C^*$-seminorm is bounded, in the construction of the universal $C^*$-algebra, see for instance \cite[Sect. 2.2]{WeLInd}.}
we say that $X$ is \emph{$n$-admissible}, and we call $A_n(X)$ the \emph{partition $C^*$-algebra associated to} $X$.
\end{defn}

\begin{lem}\label{LemAdm}
If $\paarpart\in X$ or $\baarpart\in X$, then $X$ is $n$-admissible for all $n\in\N$.
\end{lem}
\begin{proof}
Since $u_{ij}$ is self-adjoint, $u_{ij}^2$ is positive and so is $\sum_{m\neq j} u_{im}^2$. Thus, $u_{ij}^2\leq\sum_{m}u_{im}^2=1$, which proves $\|u_{ij}\|^2\leq 1$. Hence the seminorm of each generator $u_{ij}$  is bounded which shows that $A_n(X)$ exists, see also \cite[Sect. 2.2]{WeLInd}.
\end{proof}

\begin{ex}\label{ExCStar}
Here are two examples of partition $C^*$-algebras.
\begin{itemize}
\item[(a)] $A_n(\{\paarpart,\baarpart\})=C^*(u_{ij}\;|\; u_{ij}=u_{ij}^*, \sum_m u_{mi}u_{mj}=\sum_m u_{im}u_{jm}=\delta_{ij})$.
\item[(b)] $A_n(\{\paarpart,\baarpart,\dreipartrot,\downdreipartrot\})=C^*(u_{ij}\;|\; u_{ij}=u_{ij}^*=u_{ij}^2, \sum_m u_{mi}=\sum_m u_{im}=1)$. We expressed the relations of this $C^*$-algebra in a slightly different way, using that projections summing up to one are necessarily mutually orthogonal. 
\end{itemize}
The first $C^*$-algebra was defined in the context of Wang's free orthogonal quantum group \cite{WaOn} while the second one underlies his free symmetric quantum group \cite{WaSn}. See also Appendix A for more examples or \cite{We,RWfull,RWsemi}.
\end{ex}

\subsection{The relations and the category operations}
 
We will now study which relations $R(p)$ are implied by other relations $R(q)$. This will be the basis for a partition calculus on the relations $R(p)$.

\begin{defn}\label{DefH}
Let $X\subset P$ be an $n$-admissible set of partitions. We denote by $H_n(X)$ the set of all partitions $p\in P$ such that the associated relations $R(p)$ are fulfilled in $A_n(X)$.
\end{defn}

\begin{lem}\label{LemAHXGleichAX}
Let $X\subset P$ and $Y\subset P$ be two $n$-admissible sets.
\begin{itemize}
\item[(a)] We have $X\subset H_n(X)$.
\item[(b)] If $X\subset Y$, then $H_n(X)\subset H_n(Y)$.
\item[(c)] We have $A_n(X)=A_n(H_n(X))$, in the sense that there is an isomorphism between these two $C^*$-algebras mapping generators to generators.
\item[(d)] We have $H_n(X)=H(H_n(X))$.
\item[(e)] If $X\subset Y\subset H_n(X)$, then $H_n(X)=H_n(Y)$.
\end{itemize}
\end{lem}
\begin{proof}
(a) By definition of $A_n(X)$ and $H_n(X)$.

(b) Since $X\subset Y$, we have a canonical surjection from $A_n(X)$ to $A_n(Y)$ mapping the generators $u_{ij}$ of $A_n(X)$ to the generators $u_{ij}$ of $A_n(Y)$. Hence, all relations $R(p)$ holding in $A_n(X)$ also hold  in $A_n(Y)$.

(c) For $p\in H_n(X)$, the relation $R(p)$ is fulfilled in $A_n(X)$ by definition. Hence, we have a surjection from $A_n(H_n(X))$ to $A_n(X)$ mapping generators to generators. A surjection in the converse direction follows from $X\subset H_n(X)$ and hence $A_n(X)=A_n(H_n(X))$.

(d) Let $p\in H_n(H_n(X))$. Then $R(p)$ is fulfilled in $A_n(H_n(X))=A_n(X)$, hence $p\in H_n(X)$. For $H_n(X)\subset H_n(H_n(X))$, we use (a).

(e) Using (b) and (d), we have $H_n(X)\subset H_n(Y)\subset H_n(H_n(X))=H_n(X)$.
\end{proof}

The following theorem is the basis of a partition calculus for partition $C^*$-algebras, see Section \ref{SectPartCalc}.

\begin{thm}\label{mainthm}
Let $X\subset P$ be an $n$-admissible set. Let $p\in P(k,l)$ and $q\in P(k',l')$.
\begin{itemize}
\item[(a)] We have $\idpart\in H_n(X)$.
\item[(b)] If $p,q\in H_n(X)$, then also $p\otimes q\in H_n(X)$.
\item[(c)] If $p,q\in H_n(X)$ and $l=k'$, then also $qp\in H_n(X)$.
\item[(d)] If $p\in H_n(X)$, then also $\tilde p\in H_n(X)$.
\item[(e)] Let  $\paarpart,\baarpart\in H_n(X)$. If $p\in H_n(X)$, then also $p^*\in H_n(X)$.
\end{itemize}
\end{thm}
\begin{proof}
(a) This follows directly from Example \ref{ExRel}.

(b) Let $\alpha$ be a multi index of length $k$, $\alpha'$ be of length $k'$, $\beta$ be of length $l$, and $\beta'$ be of length $l'$. For the concatenated multi indices $\alpha\alpha'$ of length $k+k'$ and $\beta\beta'$ of length $l+l'$, we have:
\[\delta_{p\otimes q}(\alpha\alpha',\beta\beta')=\delta_p(\alpha,\beta)\delta_q(\alpha',\beta')\]
 Thus, using $R(p)$ and $R(q)$, we infer that $R({p\otimes q})$ is fulfilled in $A_n(X)$:
\begin{align*}
\sum_{\gamma,\gamma'}&\delta_{p\otimes q}(\gamma\gamma',\beta\beta')u_{\gamma_1\alpha_1}\ldots u_{\gamma_k\alpha_k}u_{\gamma'_1\alpha'_1}\ldots u_{\gamma'_{k'}\alpha'_{k'}}\\
&=\left(\sum_{\gamma}\delta_{p}(\gamma,\beta)u_{\gamma_1\alpha_1}\ldots u_{\gamma_k\alpha_k}\right)\left(\sum_{\gamma'}\delta_{q}(\gamma',\beta')u_{\gamma'_1\alpha'_1}\ldots u_{\gamma'_{k'}\alpha'_{k'}}\right)\\
 &= \left(\sum_{\zeta} \delta_{p}(\alpha,\zeta)u_{\beta_1\zeta_1}\ldots u_{\beta_l\zeta_l}\right)\left(\sum_{\zeta'} \delta_{q}(\alpha',\zeta')u_{\beta'_1\zeta'_1}\ldots u_{\beta'_{l'}\zeta'_{l'}}\right)\\
 &= \sum_{\zeta,\zeta'} \delta_{p\otimes q}(\alpha\alpha',\zeta\zeta')u_{\beta_1\zeta_1}\ldots u_{\beta_l\zeta_l}u_{\beta'_1\zeta'_1}\ldots u_{\beta'_{l'}\zeta'_{l'}}
\end{align*}

(c) Let $\alpha$ be a multi index of length $k$, and $\beta$ be of length $l'$. Then $\delta_{qp}(\alpha, \beta)=1$ if and only if there exists a multi index $\zeta$ of length $l$ such that $\delta_p(\alpha,\zeta)=1$ and $\delta_q(\zeta,\beta)=1$. Indeed, $\delta_{qp}(\alpha, \beta)=1$ simply means that we may label the strings of $qp$ by numbers imposed by $\alpha$ and $\beta$; the intersection of these strings with the middle points then yields a multi index $\zeta$.
The number of all possible multi indices $\zeta$ is $n^{\rl(q,p)}$, where $\rl(q,p)$ denotes the number of loops appearing throughout the procedure of composing $q$ and $p$, see Section \ref{SectOper}. Therefore, we obtain the formula:
\[\delta_{qp}(\alpha,\beta)=n^{-\rl(q,p)}\sum_\zeta\delta_p(\alpha, \zeta)\delta_q(\zeta,\beta)\]
Using $R(p)$ and $R(q)$, we deduce that $R({qp})$  is fulfilled in $A_n(X)$:
\allowdisplaybreaks
\begin{align*}
\sum_\gamma\delta_{qp}(\gamma,\beta)&u_{\gamma_1 \alpha_1}\ldots u_{\gamma_k \alpha_k}\\
&=n^{-\rl(q,p)}\sum_\gamma\sum_\zeta\delta_p(\gamma, \zeta)\delta_q(\zeta,\beta)u_{\gamma_1 \alpha_1}\ldots u_{\gamma_k \alpha_k}\\
&=n^{-\rl(q,p)}\sum_\zeta\delta_q(\zeta,\beta)\left(\sum_\gamma\delta_p(\gamma, \zeta)u_{\gamma_1\alpha_1}\ldots u_{\gamma_k\alpha_k}\right)\\
&=n^{-\rl(q,p)}\sum_\zeta\delta_q(\zeta,\beta)\left(\sum_{\zeta'}\delta_p(\alpha,\zeta')u_{\zeta_1\zeta'_1}\ldots u_{\zeta_l\zeta'_l}\right)\\
&=n^{-\rl(q,p)}\sum_{\zeta'}\delta_p(\alpha,\zeta')\left(\sum_\zeta\delta_q(\zeta,\beta)u_{\zeta_1\zeta'_1}\ldots u_{\zeta_l\zeta'_l}\right)\\
&=n^{-\rl(q,p)}\sum_{\zeta'}\delta_p(\alpha,\zeta')\left(\sum_{\gamma'}\delta_q(\zeta',\gamma')u_{\beta_1\gamma'_1}\ldots u_{\beta_{l'}\gamma'_{l'}}\right)\\
&=\sum_{\gamma'}\left(n^{-\rl(q,p)}\sum_{\zeta'}\delta_p(\alpha,\zeta')\delta_q(\zeta',\gamma')\right)u_{\beta_1\gamma'_1}\ldots u_{\beta_{l'}\gamma'_{l'}}\\
&=\sum_{\gamma'}\delta_{qp}(\alpha,\gamma')u_{\beta_1\gamma'_1}\ldots u_{\beta_{l'}\gamma'_{l'}}
\end{align*}

(d) Apply the involution to the equations of $R(p)$ of Definition \ref{DefRP} and note that \[\delta_p(\alpha_1\alpha_2\ldots\alpha_k,\beta_1\beta_2\ldots\beta_l)=\delta_{\tilde p}(\alpha_k\ldots\alpha_2\alpha_1,\beta_l\ldots\beta_2\beta_1).\]

(e) Using (d), we derive from $R(p)$ for any $\gamma$ of length $l$ and any $\gamma'$ of length $k$:
\[R({\tilde p}):\quad\sum_\zeta\delta_{\tilde p}(\zeta_k\ldots\zeta_1,\gamma_l\ldots\gamma_1)u_{\zeta_k\gamma'_k}\ldots u_{\zeta_1\gamma'_1}=\sum_\eta\delta_{\tilde p}(\gamma'_k\ldots\gamma'_1,\eta_l\ldots\eta_1)u_{\gamma_l\eta_l}\ldots u_{\gamma_1\eta_1}\]
Moreover, from $\paarpart\in H_n(X)$ and Example \ref{ExRel}, we deduce for any $\gamma$ and $\beta$ of length $k$ the formula
\[\sum_{\zeta,\gamma'}\delta_{p^*}(\gamma,\zeta)u_{\beta_1\gamma'_1}\ldots u_{\beta_k\gamma'_k}u_{\zeta_k\gamma'_k}\ldots u_{\zeta_1\gamma'_1}=\sum_\zeta\delta_{p^*}(\gamma,\zeta)\delta_{\beta\zeta}=\delta_{p^*}(\gamma,\beta)\]
while $\baarpart\in H_n(X)$ and Example \ref{ExRel} yields:
\[\sum_\gamma u_{\gamma_l\eta_l}\ldots u_{\gamma_1\eta_1}u_{\gamma_1\alpha_1}\ldots u_{\gamma_l\alpha_l}=\delta_{\eta\alpha}\]
Finally, since
\[\delta_{p^*}(\alpha,\beta)=\delta_p(\beta,\alpha)=\delta_{\tilde p}(\beta_k\ldots\beta_1,\alpha_l\ldots\alpha_1)\]
the relations $R({p^*})$ hold in $A_n(X)$ for $\alpha$  of length $l$ and $\beta$  of length $k$:
\begin{align*}
\sum_\gamma&\delta_{p^*}(\gamma,\beta)u_{\gamma_1\alpha_1}\ldots u_{\gamma_l\alpha_l}\\
&=\sum_{\gamma,\zeta,\gamma'}\delta_{p^*}(\gamma,\zeta)u_{\beta_1\gamma'_1}\ldots u_{\beta_k\gamma'_k}u_{\zeta_k\gamma'_k}\ldots u_{\zeta_1\gamma'_1}u_{\gamma_1\alpha_1}\ldots u_{\gamma_l\alpha_l}\\
&=\sum_{\gamma,\gamma'}u_{\beta_1\gamma'_1}\ldots u_{\beta_k\gamma'_k}\left(\sum_\zeta \delta_{\tilde p}(\zeta_k\ldots\zeta_1,\gamma_l\ldots\gamma_1)u_{\zeta_k\gamma'_k}\ldots u_{\zeta_1\gamma'_1}\right)u_{\gamma_1\alpha_1}\ldots u_{\gamma_l\alpha_l}\\
&=\sum_{\gamma,\gamma'}u_{\beta_1\gamma'_1}\ldots u_{\beta_k\gamma'_k}\left(\sum_\eta\delta_{\tilde p}(\gamma'_k\ldots\gamma'_1,\eta_l\ldots\eta_1)u_{\gamma_l\eta_l}\ldots u_{\gamma_1\eta_1}\right)u_{\gamma_1\alpha_1}\ldots u_{\gamma_l\alpha_l}\\
&=\sum_{\eta,\gamma'}\delta_{p^*}(\eta,\gamma')u_{\beta_1\gamma'_1}\ldots u_{\beta_k\gamma'_k}\left(\sum_\gamma u_{\gamma_l\eta_l}\ldots u_{\gamma_1\eta_1}u_{\gamma_1\alpha_1}\ldots u_{\gamma_l\alpha_l}\right)\\
&=\sum_{\eta,\gamma'}\delta_{p^*}(\eta,\gamma')u_{\beta_1\gamma'_1}\ldots u_{\beta_k\gamma'_k}\delta_{\eta\alpha}\\
&=\sum_{\gamma'}\delta_{p^*}(\alpha,\gamma')u_{\beta_1\gamma'_1}\ldots u_{\beta_k\gamma'_k}
\end{align*}
\end{proof}

The following corollary is of interest in the theory of Banica-Speicher quantum groups.

\begin{cor}\label{CorBSQG}
If $X\subset P$ is $n$-admissible and $\paarpart,\baarpart\in H_n(X)$, then $H_n(X)$ is a category of partitions in the sense of Section \ref{SectCateg}. Moreover, if $\mathcal C=\langle p_1,\ldots,p_m\rangle$ is a category of partitions generated by partitions $p_1,\ldots,p_m$ (i.e. $\mathcal C$ is the smallest category of partitions containing $p_1,\ldots,p_m$), then:
\[A_n(\mathcal C)=A_n(\{p_1,\ldots,p_m,\paarpart,\baarpart\})\]
\end{cor}
\begin{proof}
By the above theorem, $H_n(X)$ is a category of partitions for \linebreak $X:=\{p_1,\ldots,p_m,\paarpart,\baarpart\}$ and we have $X\subset\mathcal C\subset H_n(X)$.
By Lemma \ref{LemAHXGleichAX} we infer $A_n(\mathcal C)=A_n(X)$.
\end{proof}

\subsection{Partition calculus}
\label{SectPartCalc}

As a consequence of Theorem \ref{mainthm}, we obtain a partition calculus for the relations in partition $C^*$-algebras: If a partition $p\in P$ may be constructed from partitions $p_1,\ldots,p_m\in H_n(X)$ using the operations as in the above theorem, then $R(p)$ holds in $A_n(X)$. Thus, given a (typically quite small) set $X\subset P$, many more relations $R(q)$ hold in $A_n(X)$.  This turns the theory of partition $C^*$-algebras into a quite combinatorial one, see also Section \ref{SectCQG}.
In the situation when $u$ is orthogonal (i.e. when $\paarpart,\baarpart\in H_n(X)$), it is the structure of categories of partitions that underlies the relations in $A_n(X)$. Note that the partition calculus has some links to partition algebras and Temperly-Lieb algebras \cite{PartAlg}.

\begin{ex}
Coming back to the example in the introduction, we can derive the commutativity of $A_n(X)$ from the relations
\[\delta_{\beta_1\beta_4}\delta_{\beta_2\beta_3}\delta_{\beta_3\beta_5}\delta_{\beta_6\beta_8}
u_{\beta_2\alpha_1}u_{\beta_7\alpha_2}
=\sum_{\gamma'_1,\gamma'_2}u_{\beta_1\gamma'_1}u_{\beta_2\alpha_1}u_{\beta_3\alpha_1}u_{\beta_4\gamma'_1}u_{\beta_5\alpha_1}u_{\beta_6\gamma'_2}
u_{\beta_7\alpha_2}u_{\beta_8\gamma'_2}\]
and
\[\sum_\gamma u_{\gamma \alpha_1}u_{\beta_2 \alpha_2}u_{\gamma \alpha_3}u_{\beta_1 \alpha_4}u_{\beta_2 \alpha_5}=\delta_{\alpha_1\alpha_3}\delta_{\alpha_2\alpha_5}u_{\beta_1\alpha_4}u_{\beta_2\alpha_2}.\]
Indeed, these two relations are $R(p)$ and $R(q)$ with
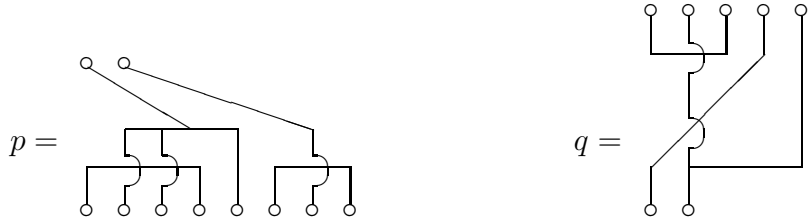
\newsavebox{\boxbogen}
   \savebox{\boxbogen}
   { \begin{picture}(4,5.5)
     \put(0,0){\line(0,1){0.5}}
     \put(0,0.9){\oval(0.8,0.8)[r]}
     \put(0,1.3){\line(0,1){0.7}}
     \end{picture}}     
\newsavebox{\boxbogenlang}
   \savebox{\boxbogenlang}
   { \begin{picture}(4,5.5)
     \put(0,0){\line(0,1){0.5}}
     \put(0,0.9){\oval(0.8,0.8)[r]}
     \put(0,1.3){\line(0,1){1}}
     \end{picture}}     
\newsavebox{\boxpm}
   \savebox{\boxpm}
   { \begin{picture}(4,5.5)
     \put(0.05,3.9){$\circ$}
     \put(1.05,3.9){$\circ$}
     \put(-1,0.35){\uppartii{1}{1}{4}}
     \put(-1,0.35){\uppartii{1}{6}{8}}
     \put(-1,0.35){\upparti{2}{5}}     
     \put(1,0.35){\usebox{\boxbogen}}
     \put(2,0.35){\usebox{\boxbogen}}
     \put(6,0.35){\usebox{\boxbogen}}    
     \put(1.3,2.35){\line(1,0){3}}
     \put(0.3,4){\line(5,-3){2.7}}      
     \put(1.3,4){\line(3,-1){5}}
     \put(0.05,0){$\circ$}
     \put(1.05,0){$\circ$}
     \put(2.05,0){$\circ$}
     \put(3.05,0){$\circ$}     
     \put(4.05,0){$\circ$}  
     \put(5.05,0){$\circ$}
     \put(6.05,0){$\circ$}
     \put(7.05,0){$\circ$}
     \end{picture}}     
\newsavebox{\boxqm}
   \savebox{\boxqm}
   { \begin{picture}(4,5.5)
     \put(0.05,5.3){$\circ$}
     \put(1.05,5.3){$\circ$}
     \put(2.05,5.3){$\circ$}
     \put(3.05,5.3){$\circ$}
     \put(4.05,5.3){$\circ$}
     \put(-1,0.35){\upparti{1}{1}}
     \put(-1,0.35){\upparti{1}{2}}
     \put(1,1.35){\usebox{\boxbogen}}
     \put(1,3.35){\usebox{\boxbogen}}
     \put(1.3,1.35){\line(1,0){3}}
     \put(0.3,1.35){\line(1,1){3}}      
     \put(-1,6.35){\partii{1}{1}{3}}
     \put(-1,6.35){\parti{1}{4}} 
     \put(-1,6.35){\parti{4}{5}}     
     \put(0.05,0){$\circ$}
     \put(1.05,0){$\circ$}
     \end{picture}}     
\newsavebox{\boxqmstretch}
   \savebox{\boxqmstretch}
   { \begin{picture}(4,5.5)
     \put(0.05,5.3){$\circ$}
     \put(1.05,5.3){$\circ$}
     \put(5.05,5.3){$\circ$}
     \put(6.05,5.3){$\circ$}
     \put(7.05,5.3){$\circ$}
     \put(-1,0.35){\upparti{1}{1}}
     \put(-1,0.35){\upparti{1}{2}}
     \put(1,1.05){\usebox{\boxbogenlang}}
     \put(1,3.35){\usebox{\boxbogen}}
     \put(1.3,1.35){\line(1,0){6}}
     \put(0.3,1.35){\line(2,1){6}}      
     \put(-1,6.35){\partii{1}{1}{6}}
     \put(-1,6.35){\parti{1}{7}} 
     \put(-1,6.35){\parti{4}{8}}     
     \put(0.05,0){$\circ$}
     \put(1.05,0){$\circ$}
     \end{picture}}     
\begin{center}
\begin{picture}(25,7.5)
 \put(0,2.5){$p=$}
 \put(1.5,0.7){\usebox{\boxpm}}
 \put(15,2.5){$q=$}
 \put(16.5,0.7){\usebox{\boxqm}}
\end{picture}
\end{center}
The following partition calculus proves that the crossing partition $\crosspart$ is in $H_n(X)$ using Theorem \ref{mainthm}. Hence, by Example \ref{ExRel}(f), $A_n(X)$ is commutative:
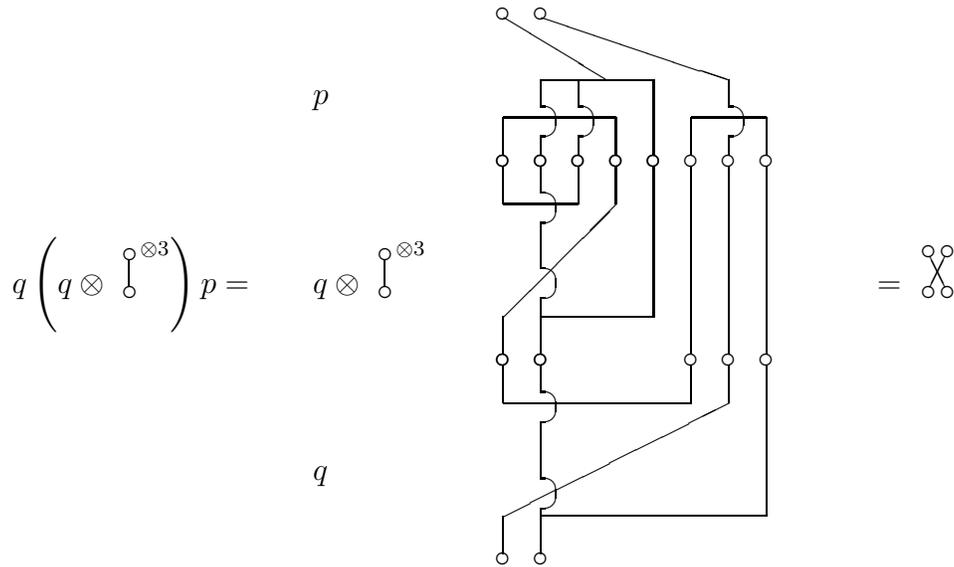
\begin{center}
\begin{picture}(26,17)
 \put(12.5,11.3){\usebox{\boxpm}}
 \put(12.5,6){\usebox{\boxqm}}
 \put(18.05,6.35){\line(0,1){5}}
 \put(19.05,6.35){\line(0,1){5}}
 \put(20.05,6.35){\line(0,1){5}} 
 \put(12.5,0.7){\usebox{\boxqmstretch}}
 \put(23,8){$= \crosspart$}
 \put(8,13){$p$}
 \put(8,8){$q\otimes\idpart^{\otimes 3}$}
 \put(8,3){$q$} 
 \put(0,8){$q\left(q\otimes\idpart^{\otimes 3}\right)p=$}
\end{picture}
\end{center}
\end{ex}

\subsection{Remarks on the links with compact quantum groups}
\label{SectCQG}

Let us very briefly sketch how  the definition of partition $C^*$-algebras goes back to the theory of \emph{Banica-Speicher quantum groups} (also called \emph{easy quantum groups}) \cite{BS,WeEQGLN,WeLInd}. A \emph{compact matrix quantum group} \cite{WoCMQG, WoCMQG2,Tim,NT} is given by a unital $C^*$-algebra $A$ generated by elements $u_{ij}$, $i,j=1,\ldots, n$ (for some $n\in\N$) such that the matrices $u=(u_{ij})$ and $\bar u=(u_{ij}^*)$ are invertible and the map
\[\Delta:A\to A\otimes_{\min} A,\qquad u_{ij}\mapsto \sum_{k=1}^n u_{ik}\otimes u_{kj}\]
is a $^*$-homomorphism. Compact matrix quantum groups form a subclass of all \emph{compact quantum groups}, as defined by Woronowicz \cite{WoCQG}. By Woronowicz's Tannaka-Krein duality theorem \cite{WoTK}, every compact matrix quantum group is abstractly equivalent to a certain  tensor category. Now, Banica and Speicher found a way to construct such tensor categories out of categories of partitions \cite{BS, TWApp}. By the duality theorem we then obtain a compact matrix quantum group, namely a Banica-Speicher quantum group. Given a Banica-Speicher quantum group, its underlying $C^*$-algebra is a partition $C^*$-algebra $A_n(X)$ with $X$ a category of partitions in the sense of Section \ref{SectCateg}, see \cite{We,WeEQGLN,WeLInd,TWop}. In this context, Corollary \ref{CorBSQG} is helpful. However, in \cite{PC2} we will also discuss  examples of partition $C^*$-algebras which do not fit into the Banica-Speicher framework.

\subsection{Reduced versions of partition $C^*$-algebras}
\label{SectRed}

If a partition $C^*$-algebra may be equiped with a quantum group structure (see Section \ref{SectCQG}), it  possesses a natural state, the Haar state, coming from Haar integration in the theory of compact quantum groups. The GNS construction with respect to this state yields a ``reduced version'' of a partition $C^*$-algebra. For instance, in the case of Example \ref{ExCStar} we  obtain non-nuclear, exact, simple $C^*$-algebras \cite{POn1,POn2,PSn}. Sometimes, it is desirable to study this reduced version rather than the universal version. One of the reasons is that these reduced versions admit nice envelopping von Neumann algebras. In Example \ref{ExCStar}(a), we obtain a strongly solid, non-injective, full, prime II$_1$ factor having the Haagerup approximation property and having no Cartan subalgebra, see \cite[Sect. 9]{WeEQGLN} for more on the von Neumann algebraic side of partition $C^*$-algebras.

\section{Questions for $C^*$-algebraists}
\label{SectOpen}

The class of partition $C^*$-algebras provides many open questions within the theory of $C^*$-algebras. They should be tractable by combinatorial means as motivated in Section \ref{SectPartCalc}.

\begin{quest}
Which sets $X\subset P$ are $n$-admissible, i.e. when do the relations $R(p)$, $p\in X$ imply that the elements $u_{ij}\in A$ are bounded independently of the $C^*$-algebra $A$? In other words: When does the universal $C^*$-algebra  $A_n(X)$ exist?
\end{quest}

This is a very basic question and apart from the case $\paarpart,\baarpart\in H_n(X)$ implying admissibility, we have no result for the moment. Finding $n$-admissible sets $X$ such that $\paarpart,\baarpart\notin H_n(X)$ would be a first step in the direction of a new kind of quantum groups, see also \cite[Sect. 7]{PC2}.

\begin{quest}\label{QuestIdeal}
What is the ideal structure of partition $C^*$-algebras? In particular, what is the kernel of the map 
\[A_n(\{\paarpart,\baarpart\})\to A_n(\{\paarpart,\baarpart,\dreipartrot,\downdreipartrot\})\]
of the $C^*$-algebras of Example \ref{ExCStar}?
\end{quest}

We know nothing about the kernel of the above map.

\begin{quest}
What is the $K$-theory of partition $C^*$-algebras? Can it be described in combinatorial terms?
\end{quest}

The $K$-groups of the $C^*$-algebras of Examples \ref{ExCStar} have been computed by Voigt \cite{Vo1,Vo2}, using  $KK$-theory and quite some quantum group machinery. See Appendix A or \cite[Thm. 9.3]{WeEQGLN} for the results. There is no direct computation available.

\begin{quest}
What are representations of partition $C^*$-algebras as operators on Hilbert spaces?
\end{quest}

We know a few representations of the $C^*$-algebras of Example \ref{ExCStar}, see \cite{Model1,Model2,Model3}. However, the pool of known representations is quite small and insufficient for many questions, in particular for distinguishing partition $C^*$-algebras up to $C^*$-isomorphisms. We are lacking of a concrete representation theory of partition $C^*$-algebras and any progress is welcome. 
Here are two concrete problems which are of particular interest for our second article building on the present one, see \cite[Sect. 7]{PC2} for more on this.

\begin{prob}
Prove or disprove that the matrix $u=(u_{ij})\in M_n(A_n(\{p,p^*\}))$ is orthogonal, for $p=\positioner$ and all $n\in\N$.
For the latter, find a number $n\in\N$, a Hilbert space $H$ and selfadjoint operators $u_{ij}\in B(H)$, $1\leq i,j\leq n$ such that:
\begin{align*}
&\textnormal{(i)} &&\sum_{k_1,k_2,k_3=1}^n u_{ak_1}u_{ik_2}u_{bk_3}u_{jk_2}
=\sum_{k_1,k_2,k_3=1}^n u_{k_1a}u_{k_2i}u_{k_3b}u_{k_2j}=\delta_{ij}\quad\textnormal{for all } a,b,i,j\\
&\textnormal{(ii)} &&\textnormal{there are $i$ and $j$ such that }\sum_{k=1}^n u_{ik}u_{jk}\neq \delta_{ij}\textnormal{ or }\sum_{k=1}^n u_{ki}u_{kj}\neq \delta_{ij}
\end{align*}
\end{prob}

\begin{prob}
Prove or disprove that the matrix $u=(u_{ij})\in M_n(A_n(\{p,p^*\}))$ is orthogonal, for $p=\crosspartoneline$ and all $n\in\N$. For the latter, it is sufficient to prove that $A_n(\{p,p^*\})$ is noncommutative, since $\paarpart,\baarpart\in H_n(\{p,p^*\})$ would imply $\crosspart\in H_n(\{p,p^*\})$ by partition calculus, and hence commutativity of $A_n(\{p,p^*\})$ by Example \ref{ExRel}(f).
Thus, find a number $n\in\N$, a Hilbert space $H$ and selfadjoint operators $u_{ij}\in B(H)$, $1\leq i,j\leq n$ such that:
\begin{align*}
&\textnormal{(i)} &&\sum_{k_1,k_2=1}^n u_{i_1k_1}u_{i_2k_2}u_{j_1k_1}u_{j_2k_2}
=\sum_{k_1,k_2=1}^n u_{k_1i_1}u_{k_2i_2}u_{k_1j_1}u_{k_2j_2}=\delta_{i_1j_1}\delta_{i_2j_2}\;\textnormal{for all } i_1,i_2,j_1,j_2\\
&\textnormal{(ii)} &&\textnormal{there are $i,j,k,l$ such that } u_{ij}u_{kl}\neq u_{kl}u_{ij}
\end{align*}
\end{prob}

Finding partition $C^*$-algebras with non-orthogonal matrices $u$ (but with a compact matrix quantum structure) would yield interesting new examples of combinatorial compact matrix quantum groups going beyond Banica and Speicher's framework, see \cite[Sect. 7]{PC2} .

\begin{quest}
Which partition $C^*$-algebras are isomorphic?
\end{quest}

A few examples of isomorphisms are known, see for instance \cite{Ra,We,RWsemi}. For the vast majority of partition  $C^*$-algebras, we know very little. This is particularly annoying when new constructions in the theory of compact quantum groups yield new examples of quantum groups -- but we cannot determine whether or not they are isomorphic to known examples of quantum groups associated to partition $C^*$-algebras. Note however, that the notion of isomorphism of compact quantum groups is stronger than an isomorphism of $C^*$-algebras.

\begin{quest}
How do partition $C^*$-algebras $A_n(X)$ depend on $n\in\N$?
\end{quest}

It is known that the $C^*$-algebra of Example \ref{ExCStar}(b) is commutative for $n=2$ and $n=3$, and non-commutative for $n\geq 4$. Some further properties of partitions $C^*$-algebras behave differently for small $n\in\N$. Moreover, by Voigt's results on $K$-theory, we know that the $C^*$-algebras of Example \ref{ExCStar} are all non-isomorphic for different $n\in\N$.



\section{Partition $C^*$-algebras: More general cases}
\label{SectGeneral}

We will now briefly sketch several more general definitions of partition $C^*$-algebras.

\subsection{The non-selfadjoint case}
\label{SectTwoCol}

The first natural generalization is the one to non-selfadjoint generators. This is also the basis of unitary Banica-Speicher quantum groups, as introduced in \cite{TWop}.

\subsubsection{The combinatorial data: two-colored set partitions}

We consider partitions  $p\in P^{\twocol}(k,l)$ on $k$ upper and $l$ lower points, where each point is either white ($\circ$) or black ($\bullet$). Like in the non-colored (or one-colored) case, we can form tensor products and compositions of partitions. Note, that the composition $qp$ of two partitions $p\in P^{\twocol}(k,l)$ and $q\in P^{\twocol}(k',l')$ is only defined if $k'=l$ and if each of the $l$ lower points of $p$ has exactly the same color as the corresponding upper point of $q$. In other words: We cannot compose black points with white points and the converse -- only white to white and black to black.

Furthermore, if $p\in P^{\twocol}(k,l)$ then its \emph{verticolor reflection} $\tilde p\in P^{\twocol}(k,l)$ is obtained by reflection at the vertical axis and inverting all colors. The adjoint $p^*\in P_{\twocol}(l,k)$ however is given by turning $p$ upside down -- and the colors remain unchanged. A category of partitions $\CC\subset P^{\twocol}$ is a collection of subsets $\CC(k,l)\subset P(k,l)$ (for all $k,l\in\N_0$) which is closed under tensor product, composition, involution and verticolor reflection containing $\paarpartwb$, $\paarpartbw$, $\baarpartwb$, $\baarpartbw$ and $\idpartww$ as well as $\idpartbb$. See \cite{TWcomb} and Appendix B.

\subsubsection{The relations associated to two-colored partitions}

Again, we define $\delta_p(i,j)$ for two multi indices $i$ and $j$ as before, regardless of the colors of $p\in P^{\twocol}(k,l)$. 
Furthermore, for all $1\leq s\leq k$ and all $1\leq t\leq l$ we define:
\[\epsilon^{\textnormal{up}}_p(s):=\begin{cases} 1 &\textnormal{if the $s$-th upper point of $p$ is white}\\ * &\textnormal{if the $s$-th upper point of $p$ is black}\end{cases}\]
and
\[\epsilon^{\textnormal{low}}_p(t):=\begin{cases} 1 &\textnormal{if the $t$-th lower point of $p$ is white}\\ * &\textnormal{if the $t$-th lower point of $p$ is black}\end{cases}\]
We then say that generators $u_{ij}$ of a unital $C^*$-algebra $A$ \emph{fulfill the relations} $R(p)$, if for all multi indices $\alpha$ and $\beta$ we have:
\[\boxed{\sum_{\gamma_1,\ldots,\gamma_k=1}^n \delta_p(\gamma,\beta) u_{\gamma_1\alpha_1}^{\epsilon^{\textnormal{up}}_p(1)}\ldots u_{\gamma_k\alpha_k}^{\epsilon^{\textnormal{up}}_p(k)}
=\sum_{\gamma_1',\ldots,\gamma_l'=1}^n \delta_p(\alpha,\gamma') u_{\beta_1\gamma_1'}^{\epsilon^{\textnormal{low}}_p(1)}\ldots u_{\beta_l\gamma_l'}^{\epsilon^{\textnormal{low}}_p(l)}}\]
Again, we use the convention for $k=0$ or $l=0$ as in Definition \ref{DefRP}.

\begin{ex}\label{ExTwocol}
Here are some examples for relations $R(p)$, see also Appendix A.
\begin{itemize}
\item[(a)] $R(\idpartwb)=R(\idpartbw)$: all $u_{ij}$ are self-adjoint.
\item[(b)] $R(\idpartbw\idpartwb)=R(\idpartwb\idpartbw)$: $u_{ij}u_{kl}^*=u_{ij}^*u_{kl}$, in particular all $u_{ij}$ are normal.
\item[(c)] The relations for the partitions $p=\paarpartwb\in P^{\twocol}(0,2)$ and $p=\baarpartbw\in P^{\twocol}(2,0)$ together are equivalent to saying that the matrix $u$ is unitary.
\item[(d)] For $p=\paarpartbw$ and $p=\baarpartwb$, we have that $\bar u:=(u_{ij}^*)$ is unitary.
\end{itemize}
\end{ex}

\subsubsection{Definition of partition $C^*$-algebras in the non-selfadjoint case}

If for $X\subset P^{\twocol}$ the unital universal $C^*$-algebra
\[\boxed{A_n(X):=C^*(1, u_{ij}, i,j=1,\ldots,n\;|\; \textnormal{the relations $R(p)$ hold for all }p\in X)}\]
exists, we say that $X$ is \emph{$n$-admissible} and we call $A_n(X)$ the \emph{partition $C^*$-algebra associated to} $X$, in analogy to Definition \ref{DefCStar}. 
Note that we may embed the class $P$ of (one-colored) partitions into the class $P^{\twocol}$ of two-colored partitions in a canonical way, by considering all points of a partition $p\in P$ as white points. Example \ref{ExTwocol}(a) then gives the link between the above definition of $A_n(X)$ and Definition \ref{DefCStar}.
Like in Lemma \ref{LemAdm}, one can show that $X$ is $n$-admissible, if one of the partitions $\paarpartwb$, $\paarpartbw$, $\baarpartwb$ and $\baarpartbw$ is in $X$.

\begin{ex}\label{ExCStar2}
As an example, consider
\[A_n(\{\paarpartwb,\paarpartbw, \baarpartwb,\baarpartbw\})=C^*(u_{ij}\;|\; u \textnormal{ and } \bar u\textnormal{ are unitary})\]
as introduced by Wang when defining his free unitary quantum group \cite{WaOn}.
\end{ex}

\begin{ex}\label{ExCStar2}
Another example is
\[A_n(\{\paarpartwb,\baarpartbw\})=C^*(u_{ij}\;|\; u \textnormal{ is unitary})\]
as introduced by Brown \cite{Unc}. Note that $\bar u$ (or equivalently $u^t$) in this $C^*$-algebra is not unitary, see \cite[Sect 4.1]{WaOn}.
\end{ex}

\subsubsection{The relations and the category operations; partition calculus}

Defining $H_n(X)\subset P^{\twocol}$ analogously to Definition \ref{DefH}, the analog of Lemma \ref{LemAHXGleichAX} holds true as well as Theorem \ref{mainthm}, when replacing $\paarpart$ and $\baarpart$ by $\paarpartwb$ and $\baarpartbw$ (or by $\paarpartbw$ and $\baarpartwb$). As in Section \ref{SectPartCalc}, we have a partition calculus (see also Appendix B).

\subsection{The non-unital case}

We now define non-unital partition $C^*$-algebras, which are linked to the so called \emph{Boolean independence} in  free probability theory. The combinatorics of Boolean independence is the one of interval partitions.

\subsubsection{The combinatorial data: interval partitions}

A partition $p\in P(0,l)$ is an \emph{interval partition}, if whenever $1\leq i<j\leq l$ are in a block of $p$, then so are all $i<s<j$. In other words, blocks in interval partitions consist only of consecutive points. Interval partitions form a subset of $P$.

\subsubsection{The relations in the non-unital situation}

We are mainly interested in defining non-unital partition $C^*$-algebras for interval partitions; however, the following definition holds for any partition $p\in P(k,l)$.
Let $A$ be a not necessarily unital $C^*$-algebra with generators $u_{ij}$ and a projection $P_0$. We say that the relations $R(p)$ are fulfilled, if all $u_{ij}$ are self-adjoint and for all multi indices $\alpha$ and $\beta$ we have:
\[P_0\sum_{\gamma_1,\ldots,\gamma_k=1}^n \delta_p(\gamma,\beta) u_{\gamma_1\alpha_1}\ldots u_{\gamma_k\alpha_k}
=\sum_{\gamma_1',\ldots,\gamma_l'=1}^n \delta_p(\alpha,\gamma') u_{\beta_1\gamma_1'}\ldots u_{\beta_l\gamma_l'}P_0\]
Again, we use the convention that the left hand side reduces to $\delta_p(\emptyset,\beta)P_0$ if $k=0$, and likewise  $\delta_p(\alpha,\emptyset)P_0$ for the right hand side in the case $l=0$.

\begin{ex}
Some examples of relations $R(p)$ are the follwing.
\begin{itemize}
\item[(a)] $R(\idpart)$: $u_{ij}$ commute with $P_0$.
\item[(b)] If $p=\paarpart\in P(0,2)$, then $R(p)$ is $\sum_k u_{ik}u_{jk}P_0=\delta_{ij}P_0$. Together with $R(p)$ for $p=\;\baarpart\;\in P(2,0)$ this amounts to Liu's definition of $P_0$-orthogonality of $u$ \cite[Def. 4.2]{Liu}.
\item[(c)] If $p=\;\crosspart\;\in P(2,2)$, then $R(p)$ does \emph{not} imply that all $u_{ij}$ commute. We only have $P_0u_{ij}u_{kl}=u_{kl}u_{ij}P_0$.
\end{itemize}
\end{ex}

\subsubsection{Definition of partition $C^*$-algebras in the non-unital case}

For $X\subset P$, we define the \emph{non-unital partition $C^*$-algebra associated to} $X$ by the not necessarily unital universal $C^*$-algebra
\[A_n(X):=C^*(P_0, u_{ij}, i,j=1,\ldots,n\;|\; \textnormal{the relations $R(p)$ hold for all }p\in X)\]
in case it exists.
Note that we easily recover all examples of relations from \cite[Def. 4.2]{Liu} apart from $P_0$-magic (i.e. turning all entries $u_{ij}$ into projections). Moreover, sending $P_0\mapsto 1$, we recover the unital situation of Section \ref{SectOnecolor}. Likewise, if $\idpart\in H_n(X)$, then $P_0$ is the unit in $A_n(X)$ and again we are in the unital situation.

\subsubsection{The relations and the category operations}

The analog of Lemma \ref{LemAHXGleichAX} holds for non-unital partition $C^*$-algebras. As for the analog of Theorem \ref{mainthm}, note that in general, $\idpart\notin H_n(X)$. Moreover, (c), (d) and (e) do not hold in general, if the $u_{ij}$ do not commute with $P_0$. Indeed, $R(p)$ and $R(q)$ are relations of the form
\[P_0X_p=Y_pP_0 \qquad \textnormal{and}\qquad P_0X_q=Y_qP_0\]
and a proof similar to Theorem \ref{mainthm}(c) would follow the scheme
\[P_0X_p=Y_pP_0=X_qP_0\]
disallowing us to conclude to $Y_qP_0$ in the end. Likewise for (d) and (e). However, $H_n(X)$ is closed under the tensor product by adapting the proof of Theorem \ref{mainthm} with respect to the following scheme:
\[P_0X_pX_q=Y_pP_0X_q=Y_pY_qP_0\]
This matches nicely with the fact that interval partitions are closed under taking the tensor product, but not under  composition. Thus, this definition of non-unital partition $C^*$-algebras is suitable for their use in the realm of quantum symmetries of Boolean independence.
Note that in principal, we have several possibilities how to insert $P_0$ into the relations $R(p)$. However, neither  $P_0X=P_0Y$, nor $XP_0=YP_0$, nor $P_0XP_0=P_0YP_0$ would yield the closure of $H_n(X)$ under taking the tensor product.

\subsection{Linear combinations of relations}

The general theory of Woronowicz's \linebreak Tannaka-Krein duality and Banica-Speicher quantum groups admits yet another generalization  of partition $C^*$-algebras. It is very likely, that they open the door to a whole new class of quantum groups -- however, we do not have a single explicit example at the moment. The difficulty begins on a purely $C^*$-algebraic level.

By $\C P(k,l)$ we denote the set of formal linear combinations $\sum_{p\in P(k,l)} a_p p$, 
where $a_p\in \C$ are some scalars. The collection of all sets $\C P(k,l)$ is denoted by $\C P$. Note that we only take linear combinations of partitions with a fixed number of upper and lower points.
The operations on $P$ extend to $\C P$ in the following way. Let $x=\sum_pa_pp\in \C P(k,l)$ and $y=\sum_qb_qq\in \C P(k',l')$.
\begin{itemize}
\item The tensor product is defined as $x\otimes y:= \sum_{p,q}a_pb_qp\otimes q$.
\item The composition is defined as $yx:= \sum_{p,q}n^{\rl(q,p)}a_pb_qqp$, if $l=k'$. The factor $\rl(q,p)$ arises naturally from Banica-Speicher's theory \cite{BS}.
\item The vertically reflected version is defined as $\tilde x:=\sum_p a_p\bar p$.
\item The involution is defined as $x^*:=\sum_p \overline{a_p} p^*$.
\end{itemize}
The relations $R(x)$  associated to $x\in\C P(k,l)$ are for self-adjoint generators $u_{ij}$:
\[\sum_p\sum_{\gamma_1,\ldots,\gamma_k=1}^n a_p\delta_p(\gamma,\beta) u_{\gamma_1\alpha_1}\ldots u_{\gamma_k\alpha_k}
=\sum_p\sum_{\gamma_1',\ldots,\gamma_l'=1}^n a_p\delta_p(\alpha,\gamma') u_{\beta_1\gamma_1'}\ldots u_{\beta_l\gamma_l'},\]
treating the cases $k=0$ or $l=0$ and defining $A_n(X)$ in the known way. It is straightforward to prove that the analogs of Lemma \ref{LemAHXGleichAX} and Theorem \ref{mainthm} hold. The following question is completely open for the moment.

\begin{quest}
Find an example of a subset $X\subset \C P$, such that $A_n(X)$ cannot be written as a partition quantum group $A_n(Y)$ for some $Y\subset P$ in the sense of Definition \ref{DefCStar}.
\end{quest}

\subsection{General coefficient case}

The most general situation in Woronowicz's theory amounts to relations of the form:
\[\sum_{\gamma_1,\ldots,\gamma_k=1}^n b(\gamma,\beta) u_{\gamma_1\alpha_1}^{\epsilon^{\textnormal{up}}_p(1)}\ldots u_{\gamma_k\alpha_k}^{\epsilon^{\textnormal{up}}_p(k)}
=\sum_{\gamma_1',\ldots,\gamma_l'=1}^n b(\alpha,\gamma') u_{\beta_1\gamma_1'}^{\epsilon^{\textnormal{low}}_p(1)}\ldots u_{\beta_l\gamma_l'}^{\epsilon^{\textnormal{low}}_p(l)},\]
for some coefficients $b(\alpha,\beta)\in\C$ and possibly non-selfadjoint generators $u_{ij}$.

\subsection{Deformations}

Two other important examples of compact matrix quantum groups are $O^+(Q)$ and $U^+(Q)$, where $Q\in GL_n(\C)$ is an invertible matrix. They were defined by Wang and Van Daele \cite{VW} and they include Woronowicz's famous $SU_q(2)$ quantum group \cite{WoSU}. Banica gave an alternative definition of $O^+(Q)$ and we follow his approach \cite{POn1}. The underlying $C^*$-algebras of $O^+(Q)$ and $U^+(Q)$ fit into the framework of partition $C^*$-algebras with a slight deformation:
\begin{align*}
&B_u(Q):=C^*(1,u_{ij}, i,j=1,\ldots,n \;|\; u \textnormal{ and } Q\bar uQ^{-1} \textnormal{ are unitary}),&&\textnormal{for general }Q\\
&B_o(Q):=C^*(1,u_{ij}, i,j=1,\ldots,n \;|\; u = Q\bar uQ^{-1} \textnormal{ is unitary}),&&\textnormal{assuming }Q\bar Q=\pm 1
\end{align*} 
For $Q=E$ the identity matrix, we recover Examples \ref{ExCStar}(a) and \ref{ExCStar2}. With \[Q=\begin{pmatrix}0&1\\-q^{-1}&0\end{pmatrix}\]
we have that $B_o(Q)$ is the $C^*$-algebra underlying Woronowicz's $SU_q(2)$.
Note that the relations of $B_u(Q)$ and $B_o(Q)$ may be seen as deformations of the relations in Example \ref{ExTwocol}(a), (c) and (d). It is not clear how to generalize it to  deformations of relations $R(p)$ for arbitrary partitions $p\in P^{\twocol}$.

\subsection{Spatial partition $C^*$-algebras}

In \cite{CeWe}, the author introduced in joint work with C\'ebron $C^*$-algebras associated with three-dimensional partitions. For  $m\in\N$ and $k,l\in\N_0$ we consider partitions of the set
\[\{1,\ldots,k,k+1,\ldots,k+l\}\times\{1,\ldots,m\}\]
into disjoint subsets. We represent such \emph{spatial partitions} $p\in P^{(m)}(k,l)$ by three dimensional pictures, see \cite[Ex. 2.3]{CeWe}. Again, we may define operations like tensor product, composition etc. Given a spatial partition $p\in P^{(m)}(k,l)$, we say that generators $u_{IJ}$ with
\[I,J\in[n_1\times\ldots\times n_m]:=\{1,\ldots,n_1\}\times\{1,\ldots,n_2\}\times\ldots\times\{1,\ldots,n_m\}\]
satisfy the relations $R(p)$, if for all tuples $A=(A_1,\ldots,A_k)$ and $B=(B_1,\ldots,B_l)$ of multi indices, we have:
\[\sum_{G_1,\ldots,G_k\in [n_1\times\ldots\times n_m]} \delta_p(G,B) u_{G_1A_1}\ldots u_{G_kA_k}=\sum_{G_1',\ldots,G_l'\in [n_1\times\ldots\times n_m]} \delta_p(A,G') u_{B_1G_1'}\ldots u_{B_lG_l'}\]
See \cite[Sect. 3.6]{CeWe} for details.

\subsection{Graph dependend partial commutativity}

In \cite{SpWe}, the author considered certain quotients of partition $C^*$-algebras, in joint work with Speicher. Given a matrix $\epsilon\in M_n(\{0,1\})$ with $\epsilon_{ii}=0$ and $\epsilon_{ij}=\epsilon_{ji}$, we define the relations $R^\epsilon$ for selfadjoint elements $u_{ij}$ by
\[u_{ik}u_{jl}=\begin{cases}
u_{jl}u_{ik}&\text{if  $\epsilon_{ij}=1$ and $\epsilon_{kl}=1$}\\
u_{jk}u_{il}&\text{if  $\epsilon_{ij}=1$ and $\epsilon_{kl}=0$}\\
u_{il}u_{jk}&\text{if  $\epsilon_{ij}=0$ and $\epsilon_{kl}=1$}
\end{cases}\]
and the relations $\mathring R^\epsilon$ by
\[u_{ik}u_{jl}=\begin{cases}
u_{jl}u_{ik}&\text{if  $\epsilon_{ij}=1$ and $\epsilon_{kl}=1$}\\
0&\text{if  $\epsilon_{ij}=1$ and $\epsilon_{kl}=0$}\\
0&\text{if  $\epsilon_{ij}=0$ and $\epsilon_{kl}=1$}
\end{cases}.\]
We then consider quotients of partition $C^*$-algebras (in particular those with respect to Example \ref{ExCStar}) by the relations $R^\epsilon$ or $\mathring R^\epsilon$. These partition $C^*$-algebras may be viewed as partially commutative partition $C^*$-algebras, where the commutativity is controlled by an undirected graph $\Gamma$ with adjacency matrix $\epsilon$. In \cite{BiAut} it is shown how the quotient of the $C^*$-algebra of Example \ref{ExCStar}(b) by the relations $\mathring R^\epsilon$ may be used to define a quantum automorphism group of $\Gamma$. See also \cite{SW} for links to quantum symmetries of graph $C^*$-algebras.

\appendix
\section*{APPENDIX A: List of $C^*$-algebraic relations}
\label{AppRel}
\setcounter{section}{1}
\setcounter{thm}{0}

\subsection{Examples of relations associated to partitions}
\label{SectAppRel}

In the following, we will deal with colored partitions $p\in P^{\twocol}(k,l)$ as in Section \ref{SectTwoCol}. Let us begin with some notation. We denote by $\singletonw\in P^{\twocol}(0,1)$ the partition consisting in a single point. We denote by 
$b_l\in P^{\twocol}(0,l)$
the partition consisting in a single block containing $l$ white lower points, so:
\[b_1=\singletonw, \qquad b_2=\paarpartww,\qquad b_3=\dreipartwww,\qquad\ldots\]
Given a partition $p\in P^{\twocol}(0,l)$ and a number $0<t<l$, we denote by
\[\rot_t(p)\in P^{\twocol}(t,l-t)\]
the partition obtained from rotating the first $t$ points of $p$ to the upper line, see Section \ref{SectAppMoreOper} for this operation.
The following list of relations associated to partitions is partially copied from \cite{TWop}. In the self-adjoint case ($u_{ij}=u_{ij}^*$) simply replace all black points by white points. 

\subsubsection{Relations built from identity partitions}

\[
R(\idpartww)=R(\idpartbb): u_{ij}=u_{ij}
\qquad\qquad\qquad\qquad\qquad\qquad\qquad\qquad\qquad\qquad\qquad\qquad\qquad\qquad\qquad
\]
\[
R(\idpartwb)=R(\idpartbw)=R(\rot_1(\paarpartww)): u_{ij}=u_{ij}^*, \textnormal{ i.e. } u=\bar u
\qquad\qquad\qquad\qquad\qquad\qquad\qquad\qquad\qquad\qquad\qquad\qquad\qquad\qquad\qquad
\]
\[
R(\idpartbw\otimes \idpartwb)=R(\idpartwb\otimes \idpartbw)=R(\rot(\paarpartww\otimes\paarpartbb)): u_{ij}^*u_{kl}=u_{ij}u_{kl}^*
\qquad\qquad\qquad\qquad\qquad\qquad\qquad\qquad\qquad\qquad\qquad\qquad\qquad\qquad\qquad
\]
\[
R(\idpartwb^{\otimes k})=R(\idpartbw^{\otimes k}): u_{i_1j_1}\ldots u_{i_lj_l}= u_{i_1j_1}^*\ldots u_{i_lj_l}^*
\qquad\qquad\qquad\qquad\qquad\qquad\qquad\qquad\qquad\qquad\qquad\qquad\qquad\qquad\qquad
\]

\subsubsection{Relations built from pair partitions}

\[
R(\paarpartww): \sum_k u_{ik}u_{jk}=\delta_{ij}, \textnormal{ i.e. } uu^t=1 
\qquad\qquad\qquad\qquad\qquad\qquad\qquad\qquad\qquad\qquad\qquad\qquad\qquad\qquad\qquad
\]
\[
R(\baarpartww): \sum_k u_{ki}u_{kj}=\delta_{ij}, \textnormal{ i.e. } u^tu=1
\qquad\qquad\qquad\qquad\qquad\qquad\qquad\qquad\qquad\qquad\qquad\qquad\qquad\qquad\qquad
\]
\[
R(\paarpartwb): \sum_k u_{ik}u_{jk}^*=\delta_{ij}, \textnormal{ i.e. } uu^*=1 
\qquad\qquad\qquad\qquad\qquad\qquad\qquad\qquad\qquad\qquad\qquad\qquad\qquad\qquad\qquad
\]
\[
R(\baarpartbw): \sum_k u_{ki}^*u_{kj}=\delta_{ij}, \textnormal{ i.e. } u^*u=1
\qquad\qquad\qquad\qquad\qquad\qquad\qquad\qquad\qquad\qquad\qquad\qquad\qquad\qquad\qquad
\]
\[
R(\paarpartbw): \sum_k u_{ik}^*u_{jk}=\delta_{ij}, \textnormal{ i.e. } \bar u(\bar u)^*=1
\qquad\qquad\qquad\qquad\qquad\qquad\qquad\qquad\qquad\qquad\qquad\qquad\qquad\qquad\qquad
\]
\[
R(\baarpartwb): \sum_k u_{ki}u_{kj}^*=\delta_{ij}, \textnormal{ i.e. } (\bar u)^*\bar u=1 
\qquad\qquad\qquad\qquad\qquad\qquad\qquad\qquad\qquad\qquad\qquad\qquad\qquad\qquad\qquad
\]
\[
R(\paarbaarpartwwww): \delta_{i'j'}\sum_k u_{ki}u_{kj}=\delta_{ij}\sum_k u_{i'k}u_{j'k}
\qquad\qquad\qquad\qquad\qquad\qquad\qquad\qquad\qquad\qquad\qquad\qquad\qquad\qquad\qquad
\]
\[
R(\paarpartww\otimes\paarpartbb): \left(\sum_k u_{ik}u_{jk}\right)\left(\sum_l u_{i'l}^*u_{j'l}^*\right)=\delta_{ij}\delta_{i'j'}
\qquad\qquad\qquad\qquad\qquad\qquad\qquad\qquad\qquad\qquad\qquad\qquad\qquad\qquad\qquad
\]

\subsubsection{Relations using crossing partitions}

\[
R(\crosspartwwww)=R(\crosspartbbbb): u_{ij}u_{kl}=u_{kl}u_{ij}
\qquad\qquad\qquad\qquad\qquad\qquad\qquad\qquad\qquad\qquad\qquad\qquad\qquad\qquad\qquad
\]
\[
R(\crosspartwbbw)=R(\crosspartbwwb): u_{ij}u_{kl}^*=u_{kl}^*u_{ij}
\qquad\qquad\qquad\qquad\qquad\qquad\qquad\qquad\qquad\qquad\qquad\qquad\qquad\qquad\qquad
\]

\subsubsection{Relations built from singleton partitions}

\[
R(\singletonw^{\otimes k}):\left(\sum_{l_1} u_{l_1j_1}\right)\ldots \left(\sum_{l_k} u_{l_kj_k}\right)=1
\qquad\qquad\qquad\qquad\qquad\qquad\qquad\qquad\qquad\qquad\qquad\qquad\qquad\qquad\qquad
\]
\[
R(\singletonw\otimes\singletonb): \left(\sum_{k} u_{kj}\right)\left(\sum_{l} u_{li}^*\right)=1
\qquad\qquad\qquad\qquad\qquad\qquad\qquad\qquad\qquad\qquad\qquad\qquad\qquad\qquad\qquad
\]
\[
R(\idpartsingletonww)=R(\idpartsingletonbb)=R(\rot_1(\singletonw\otimes\singletonb)): \left(\sum_{k} u_{kj}\right)=\left(\sum_{l} u_{il}\right)
\qquad\qquad\qquad\qquad\qquad\qquad\qquad\qquad\qquad\qquad\qquad\qquad\qquad\qquad\qquad
\]
\[
R(\rot_t(\singletonw^{\otimes s+t})): \left(\sum_{k_1} u_{k_1j_1}\right)\ldots\left(\sum_{k_s} u_{k_sj_s}\right)=\left(\sum_{l_1} u_{i_1l_1}^*\right)\ldots\left(\sum_{l_t} u_{i_tl_t}^*\right)
\qquad\qquad\qquad\qquad\qquad\qquad\qquad\qquad\qquad\qquad\qquad\qquad\qquad\qquad\qquad
\]

\subsubsection{Relations built from three blocks}

\[
R(\dreipartrot): u_{ki}u_{kj}=\delta_{ij}u_{ki}
\qquad\qquad\qquad\qquad\qquad\qquad\qquad\qquad\qquad\qquad\qquad\qquad\qquad\qquad\qquad
\]
\[
R(\downdreipartrot): u_{ik}u_{jk}=\delta_{ij}u_{ik}
\qquad\qquad\qquad\qquad\qquad\qquad\qquad\qquad\qquad\qquad\qquad\qquad\qquad\qquad\qquad
\]

\subsubsection{Relations built from four blocks}

\[
R(\vierpartrotwbwb)=R(\vierpartrotbwbw)=R(\rot_2(\vierpartwbwb)): u_{ki}u_{kj}^*=u_{ik}u_{jk}^*=0 \emph{ if }i\neq j
\qquad\qquad\qquad\qquad\qquad\qquad\qquad\qquad\qquad\qquad\qquad\qquad\qquad\qquad\qquad
\]
\[
R(\vierpartrotwwww)=R(\vierpartrotbbbb)=R(\rot_2(\vierpartwwbb)): u_{ki}u_{kj}=u_{ik}u_{jk}=0 \emph{ if }i\neq j
\qquad\qquad\qquad\qquad\qquad\qquad\qquad\qquad\qquad\qquad\qquad\qquad\qquad\qquad\qquad
\]
\[
R(\vierpartwbwb): \sum_k u_{i_1k}u_{i_2k}^*u_{i_3k}u_{i_4k}^*=\delta_{i_1i_2}\delta_{i_2i_3}\delta_{i_3i_4}
\qquad\qquad\qquad\qquad\qquad\qquad\qquad\qquad\qquad\qquad\qquad\qquad\qquad\qquad\qquad
\]
\[
R(\vierpartwwbb): \sum_k u_{i_1k}u_{i_2k}u_{i_3k}^*u_{i_4k}^*=\delta_{i_1i_2}\delta_{i_2i_3}\delta_{i_3i_4}
\qquad\qquad\qquad\qquad\qquad\qquad\qquad\qquad\qquad\qquad\qquad\qquad\qquad\qquad\qquad
\]
\[
R (b_k): \sum_l u_{i_1l}\ldots u_{i_kl}=\delta_{i_1i_2}\delta_{i_2i_3}\ldots\delta_{i_{k-1}i_k}
\qquad\qquad\qquad\qquad\qquad\qquad\qquad\qquad\qquad\qquad\qquad\qquad\qquad\qquad\qquad
\]
\[
R (\rot_d(b_d\otimes \tilde b_d)): \sum_k \delta_{i_1i_2}\delta_{i_2i_3}\ldots\delta_{i_{d-1}i_d}u_{kj_1}\ldots u_{kj_d}=\sum_l \delta_{j_1j_2}\delta_{j_2j_3}\ldots\delta_{j_{d-1}j_d}u_{i_1l}\ldots u_{i_dl}
\qquad\qquad\qquad\qquad\qquad\qquad\qquad\qquad\qquad\qquad\qquad\qquad\qquad\qquad\qquad
\]
\[
R (\rot_t(b_{s+t})): \delta_{i_1i_2}\delta_{i_2i_3}\ldots\delta_{i_{t-1}i_t}u_{i_1j_1}\ldots u_{i_1j_s}= \delta_{j_1j_2}\delta_{j_2j_3}\ldots\delta_{j_{s-1}j_s}u_{i_1j_1}^*\ldots u_{i_tj_1}^*
\qquad\qquad\qquad\qquad\qquad\qquad\qquad\qquad\qquad\qquad\qquad\qquad\qquad\qquad\qquad
\]

\subsubsection{Further relations}

\[
R(\halflibpart): u_{ij}u_{kl}u_{st}=u_{st}u_{kl}u_{ij}
\qquad\qquad\qquad\qquad\qquad\qquad\qquad\qquad\qquad\qquad\qquad\qquad\qquad\qquad\qquad
\]
Let $h_s\in P^{\twocol}(0,2s)$ be the partition consisting in two blocks, each on $s$ white lower points, the first one being placed on all odd numbers, the second one on all even numbers (see also \cite[Sect. 1.2]{RWfull}):
\[R(\rot_s(h_s)): (u_{ij}u_{kl})^s=(u_{kl}u_{ij})^s
\qquad\qquad\qquad\qquad\qquad\qquad\qquad\qquad\qquad\qquad\qquad\qquad\qquad\qquad\qquad
\]
For the fat crossing partition in $P^{\twocol}(4,4)$ consisting in two blocks on white points, each consisting in two consecutive upper points and two consecutive lower points (see also  \cite[Sect. 1.2]{RWfull}):
\[
R(\fatcross): \delta_{kk'}\delta_{ll'}u_{ki}u_{ki'}u_{lj}u_{lj'}=\delta_{ii'}\delta_{jj'}u_{lj}u_{l'j}u_{ki}u_{k'i}
\qquad\qquad\qquad\qquad\qquad\qquad\qquad\qquad\qquad\qquad\qquad\qquad\qquad\qquad\qquad
\]
For the pair positioner partition in $P^{\twocol}(3,3)$ consisting in a block of two white points (one upper, one lower) and a block of four white points (two upper, two lower) (see also  \cite[Sect. 1.2]{RWfull}):
\[
R(\legpart): \delta_{kk'}u_{ki}u_{ki'}u_{st}=\delta_{ii'}u_{st}u_{ki}u_{k'i}
\qquad\qquad\qquad\qquad\qquad\qquad\qquad\qquad\qquad\qquad\qquad\qquad\qquad\qquad\qquad
\]
The partitions consisting in singletons and pairs:
\[
R(\rot_{d+1}(\positionerd)):u_{ij}\left(\sum_{k_1} u_{k_1j_1}\right)\ldots \left(\sum_{k_d} u_{k_dj_d}\right)=\left(\sum_{l_1} u_{i_1l_1}\right)\ldots \left(\sum_{l_d} u_{i_dl_d}\right)u_{ij}
\qquad\qquad\qquad\qquad\qquad\qquad\qquad\qquad\qquad\qquad\qquad\qquad\qquad\qquad\qquad
\]
\[
R(\rot_2(\positionerwbwb)):u_{ij}\left(\sum_{k_1} u_{k_1j_1}^*\right)=\left(\sum_{l_1} u_{i_1l_1}\right)u_{ij}^*
\qquad\qquad\qquad\qquad\qquad\qquad\qquad\qquad\qquad\qquad\qquad\qquad\qquad\qquad\qquad
\]
\[
R(\rot_r(\positionerrpluseins)): u_{ij}\left(\sum_{k_1} u_{k_1j_1}\right)\ldots \left(\sum_{k_{r-1}} u_{k_{r-1}j_{r-1}}\right)
\qquad\qquad\qquad\qquad\qquad\qquad\qquad\qquad\qquad\qquad\qquad\qquad\qquad\qquad\qquad
\]
\[
\qquad\qquad\qquad\qquad\qquad\qquad\qquad\qquad\qquad\qquad\qquad
=\left(\sum_{l_1} u_{i_1l_1}\right)\ldots \left(\sum_{l_{r+1}} u_{i_{r+1}l_{r+1}}\right)u_{ij}^*
\qquad\qquad\qquad\qquad\qquad\qquad\qquad\qquad\qquad\qquad\qquad\qquad\qquad\qquad\qquad
\]

\subsection{Examples of partition $C^*$-algebras}
\label{SectAppExCStar}

\subsubsection{Free quantum orthogonal}
\label{SectAppOn}

The most prominent example of a partition $C^*$-algebra is 
\[A_n(\{\paarpart,\baarpart\})=C^*(u_{ij}\;|\; u_{ij}=u_{ij}^*, u^tu=uu^t=1)\]
underlying Wang's free orthogonal quantum group $O_n^+$, see also Example \ref{ExCStar}(a). It fulfills the Baum-Connes conjecture for quantum groups, it is $K$-amenable and its $K$-groups are $K_0(A_n(\{\paarpart,\baarpart\}))=\Z$ generated by $[1]$, and $K_1(A_n(\{\paarpart,\baarpart\}))=\Z$ generated by $[u]$, see \cite{Vo1}. Its reduced version (see Section \ref{SectRed}) is non-nuclear, exact, simple and has the metric approximation property \cite{POn1,POn2}. The commutative version of $A_n(\{\paarpart,\baarpart\})$ is the algebra $C(O_n)=A_n(\{\paarpart,\baarpart,\crosspart\})$ of continuous functions over the orthogonal group $O_n$. Moreover, one can show $A_n(\{\paarpart,\baarpart\})=A_n(NC_2)$ and $A_n(\{\paarpart,\baarpart,\crosspart\})=A_n(P_2)$ using Corollary \ref{CorBSQG}. See also \cite{We}.

\subsubsection{Free quantum unitary}
\label{SectAppUn}

The non-selfadjoint version of the above $C^*$-algebra is:
\[A_n(\{\paarpartwb,\paarpartbw, \baarpartwb,\baarpartbw\})=C^*(u_{ij}\;|\; u^*u=uu^*=1, \bar u^*\bar u=\bar u\bar u^*=1)\]
It was defined by Wang when introducing his free unitary quantum group $U_n^+$, see Example \ref{ExCStar2}. Its commutative version is the algebra of continuous functions over the unitary group:
\[A_n(\{\paarpartwb,\paarpartbw, \baarpartwb,\baarpartbw,\crosspartwwww,\crosspartwbbw\})=C(U_n)\]
The reduced version of $A_n(\{\paarpartwb,\paarpartbw, \baarpartwb,\baarpartbw\})$ and its envelopping von Neumann algebra share the same properties as the one from Section \ref{SectAppOn}, see also Section \ref{SectRed} and \cite{WeEQGLN,WeLInd}. Its $K_0$-group is $\Z$ generated by $[1]$ and its $K_1$-group is $\Z^2$ generated by $[u]$ and $[\bar u]$. See \cite{Vo3}.

\subsubsection{Free quantum symmetric}
\label{SectAppSn}

Another very important example of a partition $C^*$-algebra is
\[A_n(\{\paarpart,\baarpart,\dreipartrot,\downdreipartrot\})=C^*(u_{ij}\;|\; u_{ij}=u_{ij}^*=u_{ij}^2, \sum_m u_{mi}=\sum_m u_{im}=1)\]
arising in Wang's definition of a free symmetric quantum group $S_n^+$. Its commutative version is $C(S_n)$, where $S_n$ is the symmetric group. Again, $A_n(\{\paarpart,\baarpart,\dreipartrot,\downdreipartrot\})$ is $K$-amenable and its $K$-groups are $\Z^{n^2-2n+2}$ for $K_0$, generated by $[u_{ij}]$ for $i,j<n$ and by $[1]$, and $\Z$ for $K_1$ generated by $[u]$. See \cite{Vo2}. Moreover,\linebreak $A_n(\{\paarpart,\baarpart,\dreipartrot,\downdreipartrot\})=A_n(NC)$ and $C(S_n)=A_n(\{\paarpart,\baarpart,\dreipartrot,\downdreipartrot,\crosspart\})=A_n(P)$, see Corollary \ref{CorBSQG} and \cite{We}. Its reduced version (see Section \ref{SectRed}) is non-nuclear, exact and simple, see \cite{PSn}.

\subsubsection{From the classification of free orthogonal Banica-Speicher quantum groups}
\label{SectAppFreeOrth}

Throughout the classification of free orthogonal Banica-Speicher quantum groups, five more examples of partition $C^*$-algebras play a role, together with those from Section \ref{SectAppOn} and  \ref{SectAppSn}:
\allowdisplaybreaks
\begin{align*}
(1)
&&&A_n(\{\paarpart,\baarpart,\vierpartrot\})=C^*(u_{ij}\;|\; u_{ij}=u_{ij}^*, u^tu=uu^t=1, \\
&&&\qquad\qquad\qquad\qquad\qquad\qquad\qquad\qquad\qquad\qquad\qquad u_{ki}u_{kj}=u_{ik}u_{jk}=0 \textnormal{ if } i\neq j)\\
(2)
&&&A_n(\{\paarpart,\baarpart,\vierpartrot,\idpartsingletonww\})=C^*(u_{ij}\;|\; u_{ij}=u_{ij}^*, u^tu=uu^t=1,\\
&&&\qquad\qquad\qquad\qquad\qquad\qquad\qquad u_{ki}u_{kj}=u_{ik}u_{jk}=0 \textnormal{ if } i\neq j,\sum_k u_{ik}=\sum_k u_{kj})\\
(3)
&&&A_n(\{\paarpart,\baarpart,\singleton,\downsingleton\})=C^*(u_{ij}\;|\; u_{ij}=u_{ij}^*, u^tu=uu^t=1, \sum_k u_{ik}=\sum_k u_{kj}=1)\\
(4)
&&&A_n(\{\paarpart,\baarpart,\idpartsingletonww\})=C^*(u_{ij}\;|\; u_{ij}=u_{ij}^*, u^tu=uu^t=1, \sum_k u_{ik}=\sum_k u_{kj})\\
(5)
&&&A_n(\{\paarpart,\baarpart,\idpartsingletonww,\rot_2(\positioner)\})\\
&&&\qquad\qquad=C^*(u_{ij}\;|\; u_{ij}=u_{ij}^*, u^tu=uu^t=1, r:=\sum_k u_{ik}=\sum_k u_{kj}, ru_{ij}=u_{ij}r)
\end{align*}
Example (1) underlies Bichon's hyperoctahedral quantum group $H_n^+$ \cite{BifW}. Its commutative version are the functions over the wreath product of $\Z / 2\Z$ with $S_n$.
Example (3) gives rise to the bistochastic quantum group $B_n^+$ \cite{BS,We}. We know that (5) is isomorphic to the tensor product of (3) with $C^*(\Z / 2\Z)$, while (2) is isomorphic to a tensor product of the partition $C^*$-algebra in Section \ref{SectAppSn} with $C^*(\Z / 2\Z)$. Moreover, (4) is isomorphic to the free product of (3) with $C^*(\Z / 2\Z)$. Finally, $A_n(\{\paarpart,\baarpart,\singleton,\downsingleton\})$ is isomorphic to $A_{n-1}(\{\paarpart,\baarpart\})$ \cite{Ra,We}. The seven $C^*$-algebras from the above mentioned classification may be ordered as follows, where every map denotes the surjection mapping generators to generators, and where we abbreviate $\langle\{p_1,\ldots,p_m\}\rangle$ for $A_n(\{\paarpart,\baarpart,p_1,\ldots,p_m\})$:
\begin{displaymath}
\xymatrix{
\langle\{\singleton,\downsingleton\}\rangle\ar[d] 
&\langle\{\idpartsingletonww,\rot_2(\positioner)\}\rangle\ar[l]\ar[d] 
&\langle\{\idpartsingletonww\}\rangle\ar[l]
&\langle\emptyset\rangle\ar[l]\ar[d]\\
\langle\{\dreipartrot,\downdreipartrot\}\rangle 
&\langle\{\vierpartrot,\idpartsingletonww\}\rangle\ar[l] 
& &\langle\{\vierpartrot\}\rangle\ar[ll]
 }
\end{displaymath}
None of these seven $C^*$-algebras is exact, if $n\geq 5$, see \cite[Cor. 5.9]{We}.

\subsubsection{Further examples}

The class of orthogonal Banica-Speicher quantum groups is completely classified, see \cite{RWfull}. Some further examples of partition $C^*$-algebras showing up are:
\begin{align*}
&A_n(\{\paarpart,\baarpart,\halflibpart\})=C^*(u_{ij}\;|\; u_{ij}=u_{ij}^*, u^tu=uu^t=1,u_{ij}u_{kl}u_{st}=u_{st}u_{kl}u_{ij})\\
&A_n(\{\paarpart,\baarpart,\vierpartrot,\halflibpart\})=C^*(u_{ij}\;|\; u_{ij}=u_{ij}^*, u^tu=uu^t=1, \\
&\qquad\qquad\quad\qquad\qquad\qquad\qquad\qquad u_{ki}u_{kj}=u_{ik}u_{jk}=0 \textnormal{ if } i\neq j, u_{ij}u_{kl}u_{st}=u_{st}u_{kl}u_{ij})
\end{align*}
See also \cite[Sect. 6]{RWfull} for more examples. In \cite{TWop} many more examples with non-selfadjoint generators may be found.

\appendix
\section*{APPENDIX B: Partition Calculus}
\label{AppCalc}
\setcounter{section}{2}
\setcounter{thm}{0}

The partition calculus in Section \ref{SectPartCalc} is based on Theorem \ref{mainthm} in the sense that certain combinations of relations $R(p_1),\ldots, R(p_m)$ imply other relations $R(q)$. In other words, the strucutre of the partition calculus consists in  statements of the form:
\[p_1,\ldots,p_m\in H_n(X)\quad\Longrightarrow\quad q\in H_n(X)\]
In this appendix, we elaborate more on this partition calculus.
The following is a slight extension of an excerpt from \cite{TWcomb}. Readers interested solely in the case of partition $C^*$-algebras with selfadjoint generators should simply regard all black points as white points.

\subsection{Examples of basic operations}
\label{SectAppBasicOper}

\setlength{\unitlength}{0.5cm}
\newsavebox{\boxpexample}
   \savebox{\boxpexample}
   { \begin{picture}(3,3.5)
     \put(-1,4.35){\partii{1}{1}{2}}
     \put(-1,4.35){\partii{1}{3}{4}}
     \put(-1,0.35){\uppartii{1}{1}{2}}
     \put(-1,0.35){\upparti{2}{3}}
     \put(0.05,0){$\circ$}
     \put(1.05,0){$\circ$}
     \put(2.05,0){$\bullet$}
     \put(0.05,3.3){$\bullet$}
     \put(1.05,3.3){$\circ$}
     \put(2.05,3.3){$\bullet$}
     \put(3.05,3.3){$\circ$}
     \end{picture}}
\newsavebox{\boxqexample}
   \savebox{\boxqexample}
   { \begin{picture}(4,4.5)
     \put(-1,0.35){\upparti{4}{1}}
     \put(-1,0.35){\upparti{2}{2}}
     \put(-1,5.35){\partii{2}{2}{4}}
     \put(-1,5.35){\parti{1}{3}}
     \put(-1,5.35){\parti{1}{5}}
     \put(-1,0.35){\uppartii{1}{3}{4}}
     \put(0.05,0){$\bullet$}
     \put(1.05,0){$\circ$}
     \put(2.05,0){$\bullet$}
     \put(3.05,0){$\circ$}
     \put(0.05,4.3){$\bullet$}
     \put(1.05,4.3){$\bullet$}
     \put(2.05,4.3){$\circ$}
     \put(3.05,4.3){$\bullet$}
     \put(4.05,4.3){$\circ$}
     \end{picture}}
\newsavebox{\boxpexamplelang}
   \savebox{\boxpexamplelang}
   { \begin{picture}(3,4.5)
     \put(-1,5.35){\partii{1}{1}{2}}
     \put(-1,5.35){\partii{1}{3}{4}}
     \put(-1,0.35){\uppartii{1}{1}{2}}
     \put(-1,0.35){\upparti{3}{3}}
     \put(0.05,0){$\circ$}
     \put(1.05,0){$\circ$}
     \put(2.05,0){$\bullet$}
     \put(0.05,4.3){$\bullet$}
     \put(1.05,4.3){$\circ$}
     \put(2.05,4.3){$\bullet$}
     \put(3.05,4.3){$\circ$}
     \end{picture}}
\newsavebox{\boxpq}
   \savebox{\boxpq}
   { \begin{picture}(4,4.5)
     \put(-1,5.35){\partiii{2}{1}{2}{4}}
     \put(-1,5.35){\parti{1}{3}}
     \put(-1,5.35){\parti{1}{5}}
     \put(-1,0.35){\uppartii{1}{1}{2}}
     \put(-1,0.35){\upparti{1}{3}}
     \put(0.05,0){$\circ$}
     \put(1.05,0){$\circ$}
     \put(2.05,0){$\bullet$}
     \put(0.05,4.3){$\bullet$}
     \put(1.05,4.3){$\bullet$}
     \put(2.05,4.3){$\circ$}
     \put(3.05,4.3){$\bullet$}
     \put(4.05,4.3){$\circ$}
     \end{picture}}
\newsavebox{\boxpstern}
   \savebox{\boxpstern}
   { \begin{picture}(3,3.5)
     \put(-1,4.35){\partii{1}{1}{2}}
     \put(-1,0.35){\uppartii{1}{3}{4}}
     \put(-1,0.35){\uppartii{1}{1}{2}}
     \put(-1,0.35){\upparti{3}{3}}
     \put(0.05,3.3){$\circ$}
     \put(1.05,3.3){$\circ$}
     \put(2.05,3.3){$\bullet$}
     \put(0.05,0){$\bullet$}
     \put(1.05,0){$\circ$}
     \put(2.05,0){$\bullet$}
     \put(3.05,0){$\circ$}
     \end{picture}}
\newsavebox{\boxptilde}
   \savebox{\boxptilde}
   { \begin{picture}(3,3.5)
     \put(-1,4.35){\partii{1}{1}{2}}
     \put(-1,4.35){\partii{1}{3}{4}}
     \put(-1,0.35){\uppartii{1}{3}{4}}
     \put(-1,0.35){\upparti{2}{2}}
     \put(1.05,0){$\circ$}
     \put(2.05,0){$\bullet$}
     \put(3.05,0){$\bullet$}
     \put(0.05,3.3){$\bullet$}
     \put(1.05,3.3){$\circ$}
     \put(2.05,3.3){$\bullet$}
     \put(3.05,3.3){$\circ$}
     \end{picture}}
\newsavebox{\boxproteins}
   \savebox{\boxproteins}
   { \begin{picture}(4,3.5)
     \put(-1,4.35){\parti{1}{3}}
     \put(-1,4.35){\partii{1}{4}{5}}
     \put(-1,0.35){\uppartii{1}{2}{3}}
     \put(-1,0.35){\upparti{2}{4}}
     \put(-1,0.35){\upparti{2}{1}}
     \put(0.3,2.3){\line(1,0){2}}
     \put(0.05,0){$\circ$}
     \put(1.05,0){$\circ$}
     \put(2.05,0){$\circ$}
     \put(3.05,0){$\bullet$}
     \put(2.05,3.3){$\circ$}
     \put(3.05,3.3){$\bullet$}
     \put(4.05,3.3){$\circ$}
     \end{picture}}
\newsavebox{\boxprotzwei}
   \savebox{\boxprotzwei}
   { \begin{picture}(5,3.5)
     \put(-1,4.35){\partii{1}{5}{6}}
     \put(-1,0.35){\uppartii{1}{3}{4}}
     \put(-1,0.35){\upparti{2}{5}}
     \put(-1,0.35){\uppartii{1}{1}{2}}
     \put(0.05,0){$\bullet$}
     \put(1.05,0){$\circ$}
     \put(2.05,0){$\circ$}
     \put(3.05,0){$\circ$}
     \put(4.05,0){$\bullet$}
     \put(4.05,3.3){$\bullet$}
     \put(5.05,3.3){$\circ$}
     \end{picture}}
\newsavebox{\boxprotdrei}
   \savebox{\boxprotdrei}
   { \begin{picture}(6,2.5)
     \put(-1,0.35){\uppartiii{2}{1}{2}{7}}
     \put(-1,0.35){\uppartii{1}{3}{4}}
     \put(-1,0.35){\uppartii{1}{5}{6}}
     \put(0.05,0){$\bullet$}
     \put(1.05,0){$\circ$}
     \put(2.05,0){$\bullet$}
     \put(3.05,0){$\circ$}
     \put(4.05,0){$\circ$}     
     \put(5.05,0){$\circ$}     
     \put(6.05,0){$\bullet$}
     \end{picture}}
\newsavebox{\boxprotvier}
   \savebox{\boxprotvier}
   { \begin{picture}(6,1.5)
     \put(-1,0.35){\uppartii{1}{1}{2}}
     \put(-1,0.35){\uppartii{1}{3}{4}}
     \put(-1,0.35){\uppartiii{1}{5}{6}{7}}
     \put(0.05,0){$\bullet$}
     \put(1.05,0){$\circ$}
     \put(2.05,0){$\circ$}     
     \put(3.05,0){$\circ$}     
     \put(4.05,0){$\bullet$}
     \put(5.05,0){$\bullet$}
     \put(6.05,0){$\circ$}
     \end{picture}}
\newsavebox{\boxrexample}
   \savebox{\boxrexample}
   { \begin{picture}(3,3.5)
     \put(-1,4.35){\partii{1}{1}{2}}
     \put(-1,4.35){\partii{1}{3}{4}}
     \put(-1,0.35){\uppartii{1}{1}{2}}
     \put(-1,0.35){\upparti{1}{3}}
     \put(0.05,0){$\circ$}
     \put(1.05,0){$\circ$}
     \put(2.05,0){$\bullet$}
     \put(0.05,3.3){$\bullet$}
     \put(1.05,3.3){$\circ$}
     \put(2.05,3.3){$\bullet$}
     \put(3.05,3.3){$\circ$}
     \end{picture}}
\newsavebox{\boxrstern}
   \savebox{\boxrstern}
   { \begin{picture}(3,3.5)
     \put(-1,4.35){\partii{1}{1}{2}}
     \put(-1,4.35){\parti{1}{3}}
     \put(-1,0.35){\uppartii{1}{1}{2}}
     \put(-1,0.35){\uppartii{1}{3}{4}}
     \put(0.05,0){$\bullet$}
     \put(1.05,0){$\circ$}
     \put(2.05,0){$\bullet$}
     \put(3.05,0){$\circ$}     
     \put(0.05,3.3){$\circ$}
     \put(1.05,3.3){$\circ$}
     \put(2.05,3.3){$\bullet$}
     \end{picture}}
\newsavebox{\boxrsternr}
   \savebox{\boxrsternr}
   { \begin{picture}(3,3.5)
     \put(-1,4.35){\partii{1}{1}{2}}
     \put(-1,4.35){\partii{1}{3}{4}}
     \put(0.05,3.3){$\bullet$}
     \put(1.05,3.3){$\circ$}
     \put(2.05,3.3){$\bullet$}
     \put(3.05,3.3){$\circ$}
     \put(-1,0.35){\uppartii{1}{1}{2}}
     \put(-1,0.35){\uppartii{1}{3}{4}}
     \put(0.05,0){$\bullet$}
     \put(1.05,0){$\circ$}
     \put(2.05,0){$\bullet$}
     \put(3.05,0){$\circ$}     
     \end{picture}}     

Consider the following concrete partitions:
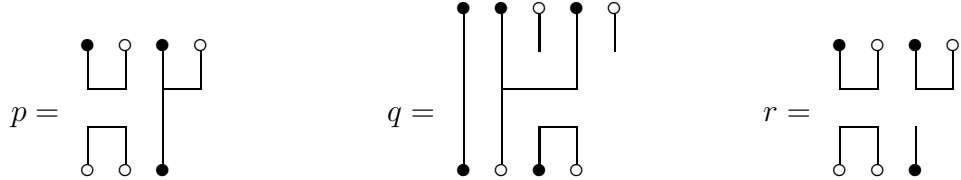
\begin{center}
\begin{picture}(25,5.5)
 \put(0,1.5){$p=$}
 \put(1.5,0){\usebox{\boxpexample}}
 \put(10,1.5){$q=$}
 \put(11.5,0){\usebox{\boxqexample}}
 \put(20,1.5){$r=$}
 \put(21.5,0){\usebox{\boxrexample}}
\end{picture}
\end{center}
Here are examples of the operations, where $\rl(p,q)$ denotes the number of removed loops throughout the composition:
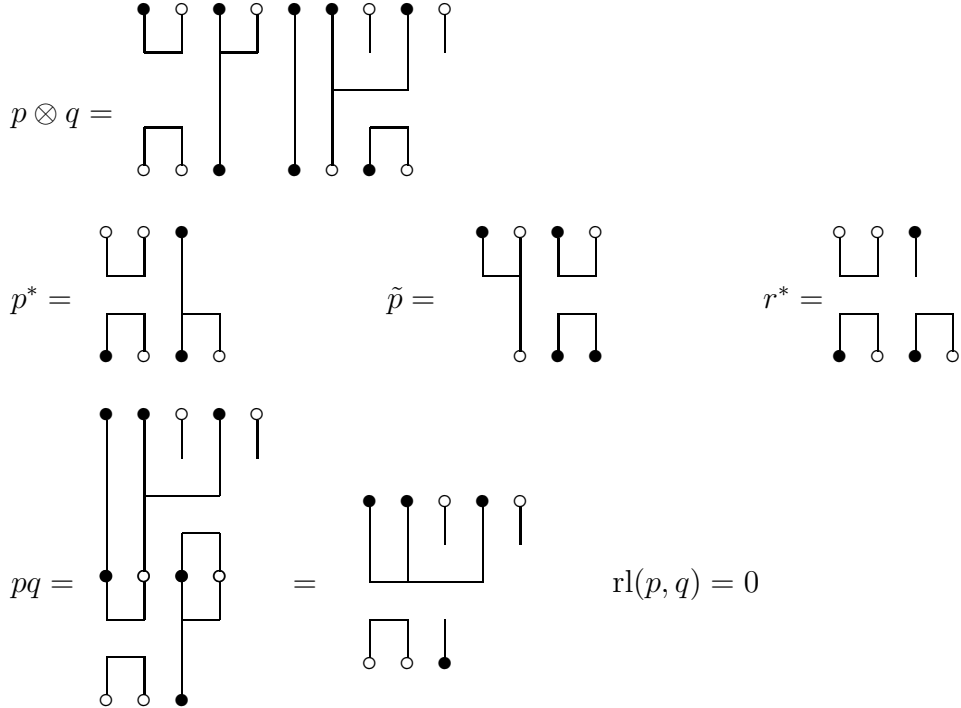
\begin{center}
\begin{picture}(25,5.5)
 \put(0,1.5){$p\otimes q=$}
 \put(3,0){\usebox{\boxpexamplelang}}
 \put(7,0){\usebox{\boxqexample}}
\end{picture}
\end{center}
\begin{center}
\begin{picture}(25,4.5)
 \put(0,1.5){$p^*=$}
 \put(2,0){\usebox{\boxpstern}}
 \put(10,1.5){$\tilde p=$}
 \put(12,0){\usebox{\boxptilde}}
 \put(20,1.5){$r^*=$}
 \put(21.5,0){\usebox{\boxrstern}} 
\end{picture}
\end{center}
\begin{center}
\begin{picture}(25,8.7)
 \put(0,3.1){$pq=$}
 \put(2,0){\usebox{\boxpexample}}
 \put(2,3.3){\usebox{\boxqexample}}
 \put(7.5,3.1){$=$}
 \put(9,1){\usebox{\boxpq}}
 \put(16,3.1){$\rl(p,q)=0$}
\end{picture}
\end{center}
\begin{center}
\begin{picture}(25,8.7)
 \put(0,3.1){$rq=$}
 \put(2,0){\usebox{\boxrexample}}
 \put(2,3.3){\usebox{\boxqexample}}
 \put(7.5,3.1){$=$}
 \put(9,1){\usebox{\boxpq}}
 \put(16,3.1){$\rl(r,q)=1$}
\end{picture}
\end{center}
\begin{center}
\begin{picture}(25,8.7)
 \put(0,3.1){$r^*r=$}
 \put(2,0){\usebox{\boxrstern}}
 \put(2,3.3){\usebox{\boxrexample}}
 \put(7.5,3.1){$=$}
 \put(9,1){\usebox{\boxrsternr}}
 \put(16,3.1){$\rl(r^*,r)=2$}
\end{picture}
\end{center}

The following lemma may be found in \cite[Lem. 1.1(d)]{TWcomb}. In fact, it holds for any set $\CC\subset P$ which is closed under tensor products and compositions and which contains the partition $\idpart$. The same holds true for generalizing Lemma \ref{LemAux}.

\begin{lem}
Let $X\subset P^{\twocol}$ be $n$-admissible and let $p\in P^{\twocol}(0,l)$ and $q\in P^{\twocol}(0,m)$. If $p,q\in H_n(X)$, then every partition obtained from placing $q$ between two legs of $p$ is in $H_n(X)$.
\end{lem}
\begin{proof}
The composition $(r_1\otimes q\otimes r_2)p$ is in $H_n(X)$ for suitable tensor products $r_1$ and $r_2$ of the identity partitions $\idpartww$ and $\idpartbb$.
\end{proof}

As an example for the above lemma:

\setlength{\unitlength}{0.5cm}
\begin{center}
\begin{picture}(28,2.5)
 \put(0,0){\uppartiv{2}{1}{3}{4}{5}}
 \put(0,0){\upparti{1}{2}}
 \put(0,0){\uppartiii{1}{7}{8}{9}}
 \put(6,0){,}
 \put(10,0){$\in H_n(X)$}
 \put(1,-0.3){$\circ$}
 \put(2,-0.3){$\bullet$}
 \put(3,-0.3){$\circ$}
 \put(4,-0.3){$\bullet$}
 \put(5,-0.3){$\bullet$}
 \put(7,-0.3){$\circ$}
 \put(8,-0.3){$\bullet$}
 \put(9,-0.3){$\bullet$}
\end{picture}
\end{center}
\begin{center}
\begin{picture}(28,5)
 \put(1,0){$\Longrightarrow$}
 \put(3,0){\uppartiv{2}{1}{6}{7}{8}}
 \put(3,0){\upparti{1}{2}}
 \put(3,0){\uppartiii{1}{3}{4}{5}}
 \put(4,-0.3){$\circ$}
 \put(5,-0.3){$\bullet$}
 \put(6,-0.3){$\circ$}
 \put(7,-0.3){$\bullet$}
 \put(8,-0.3){$\bullet$}
 \put(9,-0.3){$\circ$}
 \put(10,-0.3){$\bullet$}
 \put(11,-0.3){$\bullet$}
 \put(12,0){$=$}
 \put(12,0){\upparti{2}{1}}
 \put(12,0){\upparti{2}{2}}
 \put(12,0){\uppartiii{1}{4}{5}{6}}
 \put(12,0){\upparti{2}{8}}
 \put(12,0){\upparti{2}{9}}
 \put(12,0){\upparti{2}{10}}
 \put(12,2.5){\uppartiv{2}{1}{8}{9}{10}}
 \put(12,2.5){\upparti{1}{2}}
 \put(15,0){$\otimes$}
 \put(19,0){$\otimes$}
 \put(13,-0.3){$\circ$}
 \put(14,-0.3){$\bullet$}
 \put(16,-0.3){$\circ$}
 \put(17,-0.3){$\bullet$}
 \put(18,-0.3){$\bullet$}
 \put(20,-0.3){$\circ$}
 \put(21,-0.3){$\bullet$}
 \put(22,-0.3){$\bullet$}
 \put(13,2){$\circ$}
 \put(14,2){$\bullet$}
 \put(20,2){$\circ$}
 \put(21,2){$\bullet$}
 \put(22,2){$\bullet$}
 \put(23,0){$\in H_n(X)$}
\end{picture}
\end{center}

\subsection{Operations requiring auxiliary partitions}

If $H_n(X)$ contains certain key partitions, we have further operations inside $H_n(X)$. The following is a slight adaption of \cite[Lem. 1.3]{TWcomb}.

\begin{lem}\label{LemAux}
Let $X\subset P^{\twocol}$ be $n$-admissible and let $p\in P^{\twocol}(k,l)$.
\begin{itemize}
\item[(a)] If $\idpartwb\otimes\idpartbw,\idpartbw\otimes\idpartwb\in H_n(X)$, then $H_n(X)$ is closed under permutation of colors on the same line, i.e. if $p\in H_n(X)$, then $p'\in H_n(X)$, where $p'$ is obtained from $p$ by permutation of the colors of the points on the upper line (without changing the strings connecting the points) resp. on the lower line.
\item[(b)] If $\idpartsingletonww,\idpartsingletonbb\in H_n(X)$, then $H_n(X)$ is closed under  disconnecting any point from a block and turning it into a singleton.
\item[(c)] If $\vierpartrotwwww,\vierpartrotbbbb,\vierpartrotwbwb,\vierpartrotbwbw\in H_n(X)$, then $H_n(X)$ is closed under  connecting neighbouring blocks on the same line.
\end{itemize}
\end{lem}
\begin{proof}
The proof is an easy adaption of the proof of \cite[Lem. 1.3]{TWcomb}. We only illustrate it by pictures.

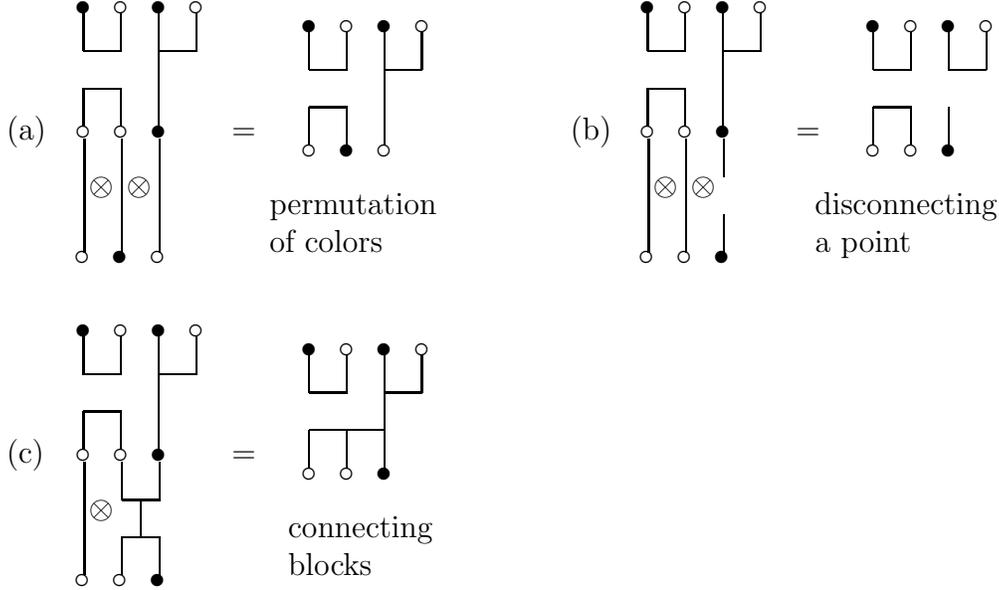
\setlength{\unitlength}{0.5cm}
\newsavebox{\boxppermuted}
   \savebox{\boxppermuted}
   { \begin{picture}(3,3.5)
     \put(-1,4.35){\partii{1}{1}{2}}
     \put(-1,4.35){\partii{1}{3}{4}}
     \put(-1,0.35){\uppartii{1}{1}{2}}
     \put(-1,0.35){\upparti{2}{3}}
     \put(0.05,0){$\circ$}
     \put(1.05,0){$\bullet$}
     \put(2.05,0){$\circ$}
     \put(0.05,3.3){$\bullet$}
     \put(1.05,3.3){$\circ$}
     \put(2.05,3.3){$\bullet$}
     \put(3.05,3.3){$\circ$}
     \end{picture}}
\newsavebox{\boxpdisconnected}
   \savebox{\boxpdisconnected}
   { \begin{picture}(3,3.5)
     \put(-1,4.35){\partii{1}{1}{2}}
     \put(-1,4.35){\partii{1}{3}{4}}
     \put(-1,0.35){\uppartii{1}{1}{2}}
     \put(-1,0.35){\upparti{1}{3}}
     \put(0.05,0){$\circ$}
     \put(1.05,0){$\circ$}
     \put(2.05,0){$\bullet$}
     \put(0.05,3.3){$\bullet$}
     \put(1.05,3.3){$\circ$}
     \put(2.05,3.3){$\bullet$}
     \put(3.05,3.3){$\circ$}
     \end{picture}}
\newsavebox{\boxpconnected}
   \savebox{\boxpconnected}
   { \begin{picture}(3,3.5)
     \put(-1,4.35){\partii{1}{1}{2}}
     \put(-1,4.35){\partii{1}{3}{4}}
     \put(-1,0.35){\uppartiii{1}{1}{2}{3}}
     \put(-1,0.35){\upparti{2}{3}}
     \put(0.05,0){$\circ$}
     \put(1.05,0){$\circ$}
     \put(2.05,0){$\bullet$}
     \put(0.05,3.3){$\bullet$}
     \put(1.05,3.3){$\circ$}
     \put(2.05,3.3){$\bullet$}
     \put(3.05,3.3){$\circ$}
     \end{picture}}
\begin{center}
\begin{picture}(28,8)
 \put(2,3){\usebox{\boxpexample}}
 \put(1.3,4){\parti{3}{1}}
 \put(1.3,4){\parti{3}{2}}
 \put(1.3,4){\parti{3}{3}}
 \put(2.3,-0.35){$\circ$}
 \put(3.3,-0.35){$\bullet$}
 \put(4.3,-0.35){$\circ$}
 \put(2.7,1.5){$\otimes$}
 \put(3.7,1.5){$\otimes$}
 \put(6.5,3){$=$}
 \put(8,2.5){\usebox{\boxppermuted}}
 \put(7.5,1){permutation}
 \put(7.5,0){of colors}
 \put(0.5,3){(a)}
 \put(17,3){\usebox{\boxpexample}}
 \put(16.3,4){\parti{3}{1}}
 \put(16.3,4){\parti{3}{2}}
 \put(16.3,4){\parti{1}{3}}
 \put(16.3,2){\parti{1}{3}}
 \put(17.3,-0.35){$\circ$}
 \put(18.3,-0.35){$\circ$}
 \put(19.3,-0.35){$\bullet$}
 \put(17.7,1.5){$\otimes$}
 \put(18.7,1.5){$\otimes$}
 \put(21.5,3){$=$}
 \put(23,2.5){\usebox{\boxpdisconnected}}
 \put(22,1){disconnecting}
 \put(22,0){a point}
 \put(15.5,3){(b)}
\end{picture}
\end{center}
\begin{center}
\begin{picture}(28,8)
 \put(2,3){\usebox{\boxpexample}}
 \put(1.3,4){\parti{3}{1}}
 \put(1.3,4){\partii{1}{2}{3}}
 \put(1.3,0){\uppartii{1}{2}{3}}
 \put(1.8,3){\parti{1}{2}}
 \put(2.3,-0.35){$\circ$}
 \put(3.3,-0.35){$\circ$}
 \put(4.3,-0.35){$\bullet$}
 \put(2.7,1.5){$\otimes$}
 \put(6.5,3){$=$}
 \put(8,2.5){\usebox{\boxpconnected}}
 \put(8,1){connecting}
 \put(8,0){blocks}
 \put(0.5,3){(c)}
\end{picture}
\end{center}
\end{proof}

\subsection{Operations requiring the pair partitions}
\label{SectAppMoreOper}

If the partitions $\paarpartwb,\paarpartbw,\baarpartwb,\baarpartbw$ are in $H_n(X)$ (for instance, if our partition $C^*$-algebra underlies a Banica-Speicher quantum group), then $H_n(X)$ is a category of partitions by Corollary \ref{CorBSQG} and we have a particularly nice partition calculus. For instance, we have a \emph{rotation} of partitions: Let $p\in P^{\twocol}(k,l)$ be a partition connecting $k$ upper points with $l$ lower points. Shifting the very left upper point to the left of the lower points  and inverting its color gives rise to a partition in $P^{\twocol}(k-1, l+1)$, a \emph{rotated version} of $p$. Note that the point still belongs to the same block after rotation. We may also rotate the leftmost lower point to the very left of the upper line (again inverting its color), and we may as well rotate in the right hand side of the lines.  In particular, for a partition $p\in P^{\twocol}(0,l)$, we may rotate the very left point to the very right and vice versa. Such a rotation on one line does \emph{not} change the colors of the points. With the example of $p$ as in  Section \ref{SectAppBasicOper}, we obtain the following rotated versions of $p$:

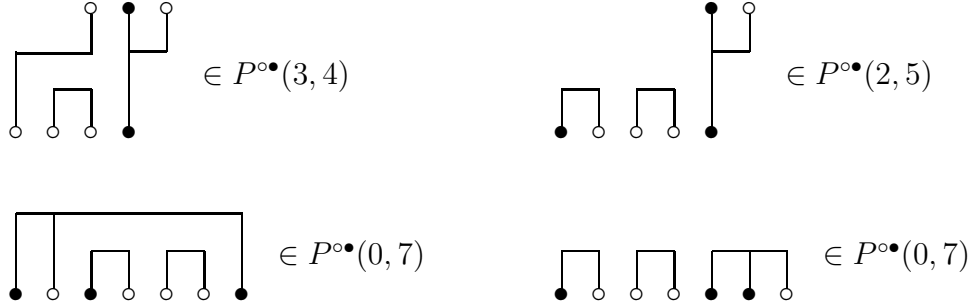
\begin{center}
\begin{picture}(25,4)
 \put(1,0){\usebox{\boxproteins}}
 \put(6.5,1.5){$\in P^{\twocol}(3,4)$}
 \put(15.5,0){\usebox{\boxprotzwei}}
 \put(22,1.5){$\in P^{\twocol}(2,5)$}
\end{picture}
\end{center}
\begin{center}
\begin{picture}(25,4)
 \put(1,0){\usebox{\boxprotdrei}}
 \put(8.5,1){$\in P^{\twocol}(0,7)$}
 \put(15.5,0){\usebox{\boxprotvier}}
 \put(23,1){$\in P^{\twocol}(0,7)$}
\end{picture}
\end{center}

The following lemma is a slight adaption of \cite[Lem. 1.1]{TWcomb}.

\begin{lem} \label{PropOper1}
Let $X\subset P^{\twocol}$ be an $n$-admissible set and let $\paarpartwb,\paarpartbw,\baarpartwb,\baarpartbw\in H_n(X)$.
\begin{itemize}
\item[(a)]  $H_n(X)$ is closed under rotation.
\item[(b)] $H_n(X)$ is closed under erasing neighbouring points of inverse colors, i.e. if $p\in P^{\twocol}(k,l)$ is a partition in $H_n(X)$, then the partition $p'\in P^{\twocol}(k,l-2)$ is in $H_n(X)$ which is obtained from $p$ by first connecting the blocks to which the $j$-th and the $(j+1)$-th lower points (of inverse colors) belong respectively, and then erasing these two points. We may also erase neighbouring points of inverse colors on the upper line.
\end{itemize}
\end{lem}
\begin{proof}
The proof may be found in \cite[Lem. 1.1]{TWcomb}. We only illustrate it here by pictures:

(a) 
\setlength{\unitlength}{0.5cm}
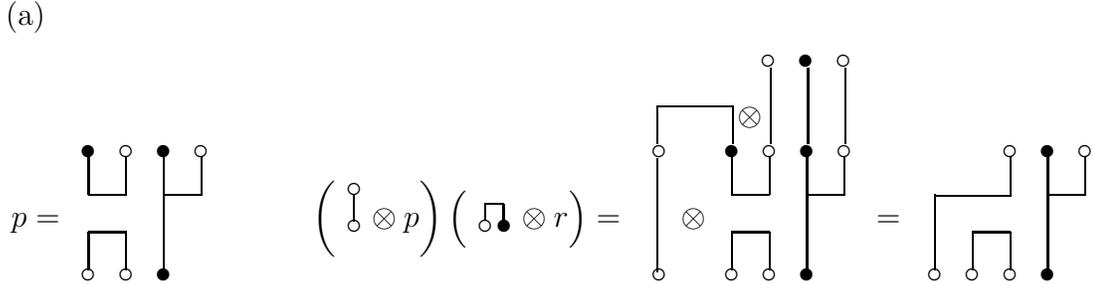
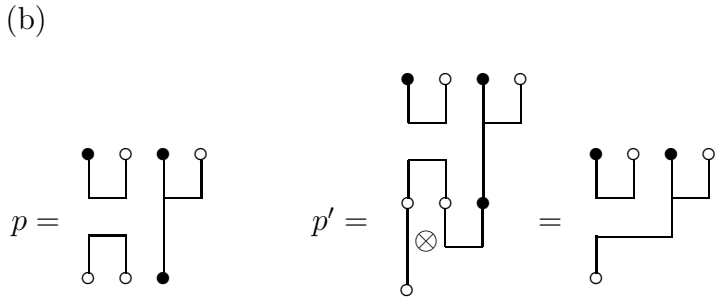
\begin{center}
\begin{picture}(28,6)
 \put(0,1.5){$p=$}
 \put(1.5,0){\usebox{\boxpexample}}
 \put(8,1.5){$\left(\idpartww\otimes p\right)\left(\paarpartwb\otimes r\right)=$}
 \put(15.9,0.3){\upparti{3}{1}}
 \put(17,0){$\circ$}
 \put(17,3.3){$\circ$}
 \put(17.8,1.5){$\otimes$}
 \put(18.6,0){\usebox{\boxpexample}}
 \put(15.9,3.7){\uppartii{1}{1}{3}}
 \put(19.3,4.2){$\otimes$}
 \put(18.9,3.7){\upparti{2}{1}}
 \put(18.9,3.7){\upparti{2}{2}}
 \put(18.9,3.7){\upparti{2}{3}}
 \put(19.9,5.7){$\circ$}
 \put(20.9,5.7){$\bullet$}
 \put(21.9,5.7){$\circ$}
 \put(23,1.5){$=$}
 \put(24,0){\usebox{\boxproteins}}
\end{picture}
\end{center}

(b) 
\setlength{\unitlength}{0.5cm}
\newsavebox{\boxpcap}
   \savebox{\boxpcap}
   { \begin{picture}(3,3.5)
     \put(-1,4.35){\partii{1}{1}{2}}
     \put(-1,4.35){\partii{1}{3}{4}}
     \put(-1,4.35){\parti{2}{3}}
     \put(-1,0.35){\upparti{1}{1}}
     \put(0.3,1.3){\line(1,0){2}}
     \put(0.05,0){$\circ$}
     \put(0.05,3.3){$\bullet$}
     \put(1.05,3.3){$\circ$}
     \put(2.05,3.3){$\bullet$}
     \put(3.05,3.3){$\circ$}
     \end{picture}}
\begin{center}
\begin{picture}(28,6)
 \put(0,1.5){$p=$}
 \put(1.5,0){\usebox{\boxpexample}}
 \put(8,1.5){$p'=$}
 \put(10,2){\usebox{\boxpexample}}
 \put(9.25,3.05){\parti{2}{1}}
 \put(10.7,1){$\otimes$}
 \put(9.25,3.05){\partii{1}{2}{3}}
 \put(10.3,-0.3){$\circ$}
 \put(14,1.5){$=$}
 \put(15,0){\usebox{\boxpcap}}
\end{picture}
\end{center}
\end{proof}

The following lemma is an adaption of \cite[Lem. 2.1]{TWcomb}.

\begin{lem}\label{LemCases}
Let $X\subset P^{\twocol}$ be $n$-admissible and let $\paarpartwb,\paarpartbw,\baarpartwb,\baarpartbw\in H_n(X)$.
\begin{itemize}
\item[(a)] $\idpartsingletonww\notin H_n(X)$ if and only if all blocks of  partitions $p\in H_n(X)$ have length at least two.
\item[(b)] $\vierpartrotwbwb\notin H_n(X)$ if and only if all blocks of partitions $p\in H_n(X)$ have length at most two.
\end{itemize}
\end{lem}
\begin{proof}
See \cite[Lem. 2.1]{TWcomb}.
\end{proof}

\bibliographystyle{alpha}

\bibliography{PartitionBib}

\end{document}